\documentclass[a4paper,12pt,reqno]{amsart}
\usepackage{amsmath}
\usepackage{amsfonts}
\usepackage{amssymb}
\usepackage{amsthm}
\usepackage{amscd}
\usepackage{eurosym}
\usepackage{hyperref}
\usepackage{verbatim}
\usepackage{eucal,url,amssymb,stmaryrd,enumerate,amscd,}
\usepackage{amsfonts}

\usepackage{verbatim}
\usepackage{color}

\newcommand{\ZZ}{\mathbb{Z}}
\newcommand{\CC}{\mathbb{C}}
\newcommand{\x}{\times}
\newcommand{\ox}{\otimes}
\newcommand{\RR}{\mathbb{R}}
\newcommand{\la}{\langle}
\newcommand{\ra}{\rangle}
\newcommand{\bd}{\partial}

\newcommand{\R}{\mathbb{R}}
\newcommand{\im}{\mathrm{im}\,}         %image of a function
\newcommand{\Z}{\mathbb{Z}}

\newcommand{\CP}{\mathbb{CP}}

\newcommand{\hook}{\lrcorner\,}
\newcommand{\SU}{\mathrm{SU}}
\newcommand{\SO}{\mathrm{SO}}

\newcommand{\Gtwo}{\mathrm{G}_2}

\newcommand{\GL}{\mathrm{GL}}
\newcommand{\SL}{\mathrm{SL}}

\newcommand{\diag}{\mathrm{diag}}
\newcommand{\too}{\longrightarrow}

\DeclareMathOperator{\Stab}{Stab}

\theoremstyle{plain}
\newtheorem{proposition}{Proposition}

\newtheorem{theorem}[proposition]{Theorem}
\newtheorem{lemma}[proposition]{Lemma}
\newtheorem{corollary}[proposition]{Corollary}
\theoremstyle{definition}
\newtheorem{definition}[proposition]{Definition}

\theoremstyle{remark}
\newtheorem{remark}[proposition]{Remark}

\renewcommand{\Re}{\mathfrak{Re}\,}
\renewcommand{\Im}{\mathfrak{Im}\,}

\begin{document}

\title[A compact $\Gtwo$-calibrated manifold with $b_1=1$]{A compact $\Gtwo$-calibrated manifold with  \\
 first Betti number $b_1=1$}

\author{Marisa Fern\'andez, Anna Fino, Alexei Kovalev and Vicente Mu\~{n}oz}

\begin{abstract}
We construct a compact, formal $7$-manifold with a closed $\Gtwo$-structure and with first Betti number $b_1\,=\,1$, which
does not admit any torsion-free $\Gtwo$-structure, that is, it does not admit any $\Gtwo$-structure such that the holonomy 
group of the associated metric is a subgroup of $\Gtwo$. We also construct
associative calibrated (hence volume-minimizing) $3$-tori with respect to
this closed $\Gtwo$-structure and, for each of those $3$-tori, we show a
$3$-dimensional family of non-trivial associative deformations.
We also construct a fibration of our $7$-manifold over 
$S^2\times S^1$ with generic fiber a (non-calibrated) coassociative $4$-torus
and some singular fibers.
\end{abstract}

\maketitle

%%%%%%%%%%%%%%%%%%%%%%%%%%%%%%%%%%%%%%%%%%%%%%%%%%%%%%%%%%%%%%%%%%%%%%%%%%%%
%%%%%%%%%%%%%%%%%%%%%%%%%%%%%%%%%%%%%%%%%%%%%%%%%%%%%%%%%%%%%%%%%%%%%%%%%%%%
%%%%%%%%%%%%%%%%%%%%%%%%%%%%%%%%%%%%%%%%%%%%%%%%%%%%%%%%%%%%%%%%%%%%%%%%%%%%
\section{Introduction}\label{sec:intro}
%%%%%%%%%%%%%%%%%%%%%%%%%%%%%%%%%%%%%%%%%%%%%%%%%%%%%%%%%%%%%%%%%%%%%%%%%%%%
%%%%%%%%%%%%%%%%%%%%%%%%%%%%%%%%%%%%%%%%%%%%%%%%%%%%%%%%%%%%%%%%%%%%%%%%%%%%

A $7$-manifold  $M$  is said to admit a $\Gtwo$-structure if there is
a reduction of the structure group of its frame bundle from  the linear group $\mathrm{GL}(7, \R)$ to the exceptional
Lie group  $\Gtwo$.  A $\Gtwo$-structure is equivalent to the
existence of a certain type of a non-degenerate $3$-form $\varphi$  (the $\Gtwo$ form)  on the manifold.  Indeed, 
by~\cite{FernandezGray} a manifold  $M$  with a  $\Gtwo$-structure comes equipped with a Riemannian metric $g$,
a cross product $P$, a $3$-form $\varphi$, and orientation, which satisfy the relation
$$
\varphi (X, Y, Z) = g(P(X,Y ),Z),
$$
for all vector fields  $X, Y, Z$ on $M$.

If the $3$-form $\varphi$ is covariantly constant with respect to the Levi-Civita
connection of the metric $g$ or, equivalently, the intrinsic torsion
of the $\Gtwo$-structure vanishes \cite{Salamon}, then the holonomy  group of $g$  is contained in $\Gtwo$, 
and the $3$-form  $\varphi$ is closed and coclosed~\cite{FernandezGray}. 
 In this case, the $\Gtwo$-structure is said to be {\em torsion-free}.
The first complete examples of  metrics with holonomy $\Gtwo$ were obtained by Bryant and Salamon in \cite{BryantSalamon},
while compact examples of Riemannian manifolds with holonomy  $\Gtwo$ were
constructed first  by Joyce \cite{Joyce1}, and then by Kovalev \cite{Kovalev}, Kovalev and Lee \cite{KovalevLee},
and Corti, Haskins, Nordstr\"om, Pacini \cite{CHNP}.
More recently, a new construction of compact manifolds with holonomy $\Gtwo$ has been given in \cite{JoyceKarigiannis} by gluing families of Eguchi--Hanson spaces.

A $\Gtwo$-structure is called {\em closed}, or {\em calibrated}, if the $3$-form $\varphi$ is closed \cite{HarveyLawson},
and a $\Gtwo$-structure is said to be {\em coclosed}, or {\em cocalibrated}, if the $3$-form  $\varphi$ is coclosed.
These two classes of $\Gtwo$-structures are very different in nature, 
the closed  condition of the $\Gtwo$ form being much more restrictive;
for example, 
Crowley and Nordstr\"om in \cite{CN} prove that coclosed $\Gtwo$-structures always exist on closed spin manifolds and 
satisfy the parametric $h-$principle.

Recently, for a compact $7$-manifold $M$ endowed with a closed non-parallel $\Gtwo$-structure, 
 Podest\`{a} and Raffero in \cite{PR} have proved that the identity component  of the automorphism group of $M$ is Abelian with dimension bounded by min$\{6,b_{2}(M)\}$.

Compact $\Gtwo$-calibrated manifolds have interesting curvature properties. 
It is well known that a $\Gtwo$ holonomy manifold is Ricci-flat or, equivalently, both Einstein and scalar-flat. 
On a compact $\Gtwo$-calibrated  manifold, both the Einstein condition \cite{CleytonIvanov} 
and scalar-flatness \cite{Bryant} are equivalent to the holonomy being contained in $\Gtwo$. 
In fact, Bryant in \cite{Bryant} shows that the scalar curvature is always non-positive.

All the previously known examples in the literature of compact $7$-manifolds
which are not a product
of $S^1$ and a
symplectic half-flat $6$-manifold in the sense of~\cite{chiossi-salamon}
admitting a closed $\Gtwo$ form, 
which is not also coclosed, have first Betti number strictly bigger than one.
The first example of a compact $\Gtwo$-calibrated manifold that  does not have
holonomy contained in $\Gtwo$ was obtained 
in \cite{Fernandez}. This example is a nilmanifold, that is
 a compact quotient of  a simply connected nilpotent  Lie group by a lattice, endowed with an invariant closed $\Gtwo$-structure.
In~\cite{CF} Conti and the first author classified the nilpotent 7-dimensional 
Lie algebras that admit a closed $\Gtwo$-structure. 
All those examples are non-formal. Other examples were given in \cite{Fernandez-2}.
They are formal compact solvable manifolds with first Betti number $b_1=3$.

 In this paper, we construct a compact formal 7-manifold with a closed $\Gtwo$-structure and with first Betti number $b_1=1$ not admitting
any  torsion-free $\Gtwo$-structure. To our knowledge, this manifold is the first example of compact
$\Gtwo$-calibrated manifold that satisfies all these properties and it is not a product. 

To construct such a manifold, we start with a compact 7-manifold $M$ equipped with 
a closed $\Gtwo$ form $\varphi$ and with first Betti number $b_{1}(M)=3$.
Then we quotient $M$ by a finite group preserving $\varphi$ to obtain an orbifold $\widehat{M}$
with a closed orbifold $\Gtwo$ form $\widehat{\varphi}$ and with first Betti number $b_{1}(\widehat{M})=1$ (Proposition \ref{prop:7-orbifold}).
We resolve the singularities of the $7$-orbifold $\widehat{M}$ to produce a smooth $7$-manifold
$\widetilde{M}$ with a closed $\Gtwo$ form $\widetilde{\varphi}$, with first Betti number $b_{1}(\widetilde{M})=1$ and such that
$(\widetilde{M}, \widetilde{\varphi})$ is 
isomorphic to $(\widehat{M}, \widehat{\varphi})$ outside 
the singular locus of $\widehat{M}$ (Theorem \ref{th:resolution}).
The idea of this construction stems from our study of Joyce's original techniques of ``$\Gtwo$-orbifold resolutions" \cite{Joyce1, Joyce2}
that enabled him to construct compact Riemannian manifolds with holonomy $\Gtwo$. 
(There ``$\Gtwo$-orbifold" means an orbifold with a closed and coclosed
orbifold $\Gtwo$ form.)

Next, we prove that $\widetilde{M}$ has the aforementioned properties. More precisely, using the concept of 3-formal minimal model, 
introduced in \cite{FernandezMunoz} as an extension of formality \cite{DGMS} (see Section \ref{sec:formality} for details) we prove that the
$7$-manifold $\widetilde{M}$ is formal (Proposition \ref{prop:formresolut}). On the other hand, we show that
$\widetilde M$ has fundamental group $\pi_1(\widetilde M) = \Z$ (Proposition \ref{prop:fundamental-group}),
this resulting from the careful choice of the action of the finite group acting on $M$. Finally, using this last result and that $b_{1}(\widetilde{M})=1$,
we prove that 
if $\widetilde{M}$ carries a $\Gtwo$ form such that the holonomy group of the associated metric 
is a subgroup of $\Gtwo$, then $\widetilde{M}$ 
has a finite covering which is a product of a 6-dimensional simply connected Calabi--Yau manifold and a circle, and so
there exist a closed $2$-form $\omega$ and a closed $1$-form $\eta$ on $\widetilde{M}$ such that 
$\omega^3 \wedge \eta\not= 0$ at every point of $\widetilde{M}$. 
But we see that this is not possible by the  cohomology of $\widetilde{M}$ determined in Proposition \ref{prop-cohom}.
This shows that $\widetilde{M}$ does not admit any torsion-free $\Gtwo$-structure (Theorem \ref{th:notorfree}).

Now, let us recall that for each $7$-manifold $N$ with a $\Gtwo$-structure $\phi$, one may define a
special class of $3$-dimensional orientable submanifolds of~$N$ called
{\em associative $3$-folds} (see section \ref{sect:3-folds} for details). Their tangent spaces are subalgebras of the cross-product
algebras induced by $\phi$ on the tangent spaces of~$N$; in fact, these
latter subalgebras are isomorphic to $\RR^3$ with the standard vector
product. If the $\Gtwo$-structure $\phi$ is closed, then $\phi$ is a calibration 
and every associative $3$-fold is a minimal submanifold of $N$ (moreover,
locally volume-minimizing in its homology class \cite[Proposition 3.7.2]{Joyce2}).

For the compact $7$-manifold $M$ with the closed $\Gtwo$ form $\varphi$ mentioned above, 
we consider a non-trivial involution of $M$ preserving $\varphi$, and we 
construct an example of a $3$-dimensional family of associative
volume-minimizing $3$-tori in $\widetilde{M}$ (Proposition
\ref{prop:deformations 3-folds}).
This deformation family is ``maximal'' (Corollary~\ref{cor:maximal}). On the
other hand, we show in Proposition~\ref{prop:isolate.associative} that
each associative $3$-torus fixed by the above involution becomes rigid and
isolated after an arbitrary small closed perturbation of the ambient $\Gtwo$-structure.

For a $G_2$-structure $\phi$ (not necessarily closed or coclosed) on a $7$-manifold $N$, we have another natural class of orientable 
submanifolds of $N$: the so-called {\em coassociative $4$-folds}. 
Such a submanifold may be
defined by the vanishing of $\phi$ (see section \ref{sect:4-folds}). When the
$\Gtwo$-structure $\phi$ is closed, the space of deformations of a coassociative
$4$-fold $X$ is a smooth manifold of dimension equal to the positive part
of the second Betti number $b^2_+(X)$. If also $b^2_+(X)=3$, then the
deformations of $X$ may `fill out' an open set in the ambient
$\Gtwo$-manifold. We construct a smooth fibration map
$\widetilde{M}\to S^2\times S^1$ with generic fiber a coassociative torus and
some singular fibers, with both smooth and singular fibers forming maximal
deformation families (Proposition~\ref{prop:fibration}).

\subsection*{Acknowledgements}
The authors would like to thank  Simon Donaldson, Dominic Joyce, 
Luc\'{\i}a Mart\'{\i}n-Merch\'an and Goncalo Oliveira for useful comments. 
The first author was partially supported by MINECO Grant PGC2018-098409-B-100
and Gobierno Vasco Grant IT1094-16, Spain. The second author  is supported
by the project   PRIN 2017  \lq \lq Real and complex manifolds: Topology, Geometry and Holomorphic Dynamics"  and by GNSAGA. of I.N.d.A.M.. 
The fourth author was partially supported by MINECO-FEDER Grant (Spain)  MTM2015-63612-P.

%%%%%%%%%%%%%%%%%%%%%%%%%%%%%%%%%%%%%%%%%%%%%%%%%%%%%%%%%%%%%%%%%%%%%%%%%%%%
%%%%%%%%%%%%%%%%%%%%%%%%%%%%%%%%%%%%%%%%%%%%%%%%%%%%%%%%%%%%%%%%%%%%%%%%%%%%
%%%%%%%%%%%%%%%%%%%%%%%%%%%%%%%%%%%%%%%%%%%%%%%%%%%%%%%%%%%%%%%%%%%%%%%%%%%%
\section{Orbifolds} \label{sec:orbifolds}
%%%%%%%%%%%%%%%%%%%%%%%%%%%%%%%%%%%%%%%%%%%%%%%%%%%%%%%%%%%%%%%%%%%%%%%%%%%%
%%%%%%%%%%%%%%%%%%%%%%%%%%%%%%%%%%%%%%%%%%%%%%%%%%%%%%%%%%%%%%%%%%%%%%%%%%%%

In this section we collect some basic facts and definitions concerning $\Gtwo$ forms on smooth manifolds and on orbifolds
(see \cite{Adem, Bryant-1,Bryant,Donaldson,FernandezGray, HarveyLawson,Hitchin,Hitchin:StableForms,Hitchin:Clay,Joyce2, Salamon} for details).

Let us consider the space $\mathbb {O}$ of the Cayley numbers (or octonions) which
is a non-associative algebra over $\RR$ of dimension $8$. We can identify
${\RR}^7$ with the subspace of  $\mathbb {O}$ consisting of pure imaginary Cayley numbers. Then,
the product on $\mathbb {O}$ defines on  ${\RR}^7$ the $3$-form $\varphi_{0}$ given by
\begin{equation} \label{eqn:Gtwoform}
\varphi_{0}= e^{127}+e^{347}+e^{567}+e^{135} -e^{236}-e^{146}-e^{245},
\end{equation}
where $\{e^1,\dotsc, e^7\}$ is the standard basis of $(\RR^7)^*$.
Here, $e^{127}$ stands for $e^1\wedge e^2\wedge e^7$, and so on.
The stabilizer of $\varphi_{0}$ under the standard action of $\GL(7,\RR)$ on $\Lambda^3(\RR^7)^*$ is 
the Lie group $\Gtwo$, which is one of the exceptional Lie groups, and it is a compact, connected, simply connected,
simple Lie subgroup of $\SO(7)$ of dimension $14$.

Note that $\Gtwo$ acts irreducibly on $\RR^7$ and preserves the standard metric and orientation for which
$\{e_1,\dotsc, e_7\}$ is an oriented and orthonormal basis.
The $\GL(7,\RR)$-orbit of $\varphi_{0}$ is open in $\Lambda^3(\RR^7)^*$, so $\varphi_{0}$ is a {\em stable} $3$-form on $\RR^7$
\cite{Hitchin}.

\begin{definition} \label{def:G2form-vector-sp-1}
Let $V$ be a real vector space of dimension $7$.  A $3$-form $\varphi\in\Lambda^3(V^*)$ on $V$
is a $\Gtwo$ {\em form} (or $\Gtwo$-{\em structure}) on $V$
if there is a linear isomorphism $u\colon (V, \varphi) \too (\RR^7,\varphi_{0})$ such that $u^{*}\varphi_{0}=\varphi$, where
$\varphi_0$ is given by \eqref{eqn:Gtwoform}.
\end{definition}

A  $\Gtwo$-{\em structure} on a $7$-dimensional smooth manifold $M$ 
is a reduction of the structure group of 
its frame bundle from ${\GL}(7,\mathbb{R})$ to the  exceptional Lie group $\Gtwo$.  
Gray in \cite{Gray} proved that a smooth $7$-manifold $M$ carries $\Gtwo$-structures
if and only if $M$ is orientable and spin.

The presence of a $\Gtwo$-structure is equivalent to the existence of a differential $3$-form $\varphi$ (the $\Gtwo$ form) on $M$, which
can be defined as follows. Denote by $T_{p}(M)$ the tangent space to $M$ at $p\in M$, and by $\Omega^*(M)$ the algebra of 
the differential forms on $M$.

\begin{definition} \label{def:G2-mfd}
Let $M$ be a smooth manifold of dimension $7$. A $\Gtwo$ {\em form} on $M$ 
is a differential $3$-form $\varphi\in \Omega^3(M)$ such that,  for each point $p\in M$, 
$\varphi_{p}$ is a $\Gtwo$ form on $T_{p}(M)$ (in the sense of Definition \ref{def:G2form-vector-sp-1})  that is, for each $p\in M$, 
 there is a linear isomorphism $u_{p}\colon (T_{p}(M), \varphi_{p}) \too (\RR^7,\varphi_{0})$ satisfying $u_{p}^{*}\varphi_{0}=\varphi_{p}$,
 where $\varphi_0$ is given by \eqref{eqn:Gtwoform}.
\end{definition}

Therefore, if  $\varphi$ is a  $\Gtwo$ form on $M$, then $\varphi$
can be locally written as \eqref{eqn:Gtwoform} with respect to some (local)
 basis $\{e^1,\dotsc, e^7\}$ of the (local) $1$-forms on $M$.

Note that there is a $1$-$1$ correspondence between  $\Gtwo$-structures and $\Gtwo$ forms on $M$.
In fact, if $\varphi\in \Omega^3(M)$ is a $\Gtwo$ form on $M$, the subbundle of the frame bundle whose 
fiber at $p\in M$  consists of the isomorphisms $u_{p}\colon (T_{p}(M), \varphi_{p}) \too (\RR^7,\varphi_{0})$, such that
$u_{p}^{*}\varphi_{0}=\varphi_{p}$, defines a principal subbundle with fiber $\Gtwo$, that is a $\Gtwo$-structure on $M$.
 
Since $\Gtwo \subset \SO(7)$, a $\Gtwo$ form on $M$ determines a Riemannian metric and an orientation on $M$.
Let $\varphi$ be a  $\Gtwo$ form on $M$. Denote by  $g_{\varphi}$ the Riemannian metric induced by $\varphi$, and
by $\nabla_{\varphi}$ the Levi-Civita connection of $g_{\varphi}$.
Let $\star_{\varphi}$ be the Hodge star operator determined by $g_{\varphi}$ and the orientation induced by $\varphi$.

\begin{definition}\label{def:closedG2-mfds}
We say that a manifold $M$ has a {\em closed} $\Gtwo$-{\em structure}
if there is a $\Gtwo$ form $\varphi$ on $M$ such that $\varphi$ is closed, that is 
$d\varphi=0$. A manifold $M$ has a {\em coclosed} $\Gtwo$-{structure}
if there is a $\Gtwo$ form $\varphi$ on $M$ such that $\varphi$ is coclosed, i.e. $d(\star_{\varphi}\varphi)=0$.
A $\Gtwo$ form $\varphi$ on $M$ is {\em torsion-free} if
$\nabla_{\varphi}\varphi=0$ (equivalently if the
$\Gtwo$-structure is closed and coclosed~\cite{FernandezGray}).
\end{definition}

\subsection*{Orbifold $\Gtwo$ forms}\label{subsect:orbif}
\begin{definition}
A (smooth) $n$-dimensional orbifold is a Hausdorff, paracompact topological 
 space $X$ endowed with an atlas $\{(U_p, {\widetilde U}_p, \Gamma_p, f_p)\}$ of orbifold charts,
that is $U_p\subset X$ is a neighbourhood of $p\in X$, $\widetilde U_p\subset \R^n$ an
open set, $\Gamma_p\subset \GL(n, \R)$ a f{}inite group acting on $\widetilde U_p$,
and $f_p\colon{\widetilde U_p}\to U_p$ is a $\Gamma_p$-invariant map
with $f_p(0)=p$, inducing a homeomorphism ${\widetilde U_p}/\Gamma_p \cong U_p$.
Moreover, the charts are compatible in the following sense:

If $q \in U_q\cap U_p$, then there exist a connected 
neighbourhood $V\subset U_q \cap U_p$ and a 
dif{}feomorphism $F\colon f_p^{-1}(V)_0 \to f_q^{-1}(V)$,
where $f_p^{-1}(V)_0$ is a 
connected component of $f_p^{-1}(V)$, 
such that $F(\sigma(x))=\rho(\sigma)(F(x))$, for
any $x$, and $\sigma\in \Stab_{\Gamma_p}(q)$, where 
$\rho\colon \Stab_{\Gamma_p}(q) \to \Gamma_q$ is a group isomorphism.
\end{definition}

For each $p\in X$, let $n_p=\# \Gamma_p$ be the order of the orbifold point (if $n_p=1$ the point is smooth, 
also called a non-orbifold point). The singular locus of the orbifold is the set $S=\{p\in X \ | \ n_p>1\}$. Therefore
 $M-S$ is a smooth $n$-dimensional manifold.
The singular locus $S$ is stratif{}ied: if we write $S_k=\{p \ | \ n_p=k\}$, and consider its closure $\overline{S_k}$, 
then $\overline{S_k}$ inherits the structure of an orbifold. In particular $S_k$ is a smooth manifold, 
and the closure consists of some points of $S_{l}$, $l\geq 2$.

We say that the orbifold is {\em locally oriented} if $\Gamma_p\subset \GL_+(n,\R)$ for any $p\in X$. As $\Gamma_p$ is f{}inite, 
we can choose a metric on $\widetilde U_p$ such that $\Gamma_p\subset 
\SO(n)$. An element $\sigma\in \Gamma_p$ admits a basis in which it is written as
\[
\sigma=\diag \left(
\left( \begin{array}{cc} \cos \theta_1 & -\sin \theta_1\\ \sin \theta_1 & \cos \theta_1\end{array} \right),
\ldots, \left( \begin{array}{cc}  \cos \theta_r & -\sin \theta_r\\ \sin \theta_r & \cos \theta_r\end{array}\right), 1,\ldots, 1\right),
\]
for $\theta_1,\ldots,\theta_r\in (0,2\pi)$. In particular, the set of points f{}ixed by $\sigma$ is of codimension $2r$. 
Therefore the set of singular points $S\cap U_p$ is of codimension $\geq 2$, and hence $X-S$ is 
connected (if $X$ is connected). Also we say that the orbifold $X$ is {\em oriented} if it is locally oriented and $X-S$ is oriented.

A natural example of orbifold appears when we take a smooth manifold $M$ and a f{}inite group $\Gamma$ acting on $M$ 
smoothly and ef{}fectively. Then $\widehat{M}=M/\Gamma$ is an orbifold. If $M$ is oriented and the action of $\Gamma$ 
preserves the orientation, then $\widehat{M}$ is an oriented orbifold. Note that for every $\widehat p\in \widehat{M}$, 
the group $\Gamma_{\widehat p}$ is the stabilizer of $p\in M$, with $\widehat p=\widehat\pi(p)$ under 
the natural projection $\widehat \pi\colon M\to \widehat{M}$.

Let $X$ be an orbifold of dimension $n$. An {\em orbifold} $k$-form $\alpha$ on $X$
consists of a collection of differential $k$-forms $\alpha_p$ $(p\in X)$ on each open $\widetilde U_p$ 
which are $\Gamma_p$-invariant and that match under
the compatibility maps between dif{}ferent charts. 
 
The space of orbifold $k$-forms on $X$ is denoted by
$\Omega_{orb}^{k}(X)$. The wedge product of orbifold forms and the 
exterior differential $d$ on $X$ are well defined. Thus, we have 
$$
d\colon\Omega_{orb}^k(X) \,\too\, \Omega_{orb}^{k+1}(X)\, .
$$
The cohomology of $(\Omega_{orb}^k(X),d)$ is the cohomology of the topological
space $X$ with real coefficients, $H^*(X)$ (see \cite[Proposition 2.13]{CFM}).

\begin{remark}
Suppose that $X=M/\Gamma$ is an orbifold, where $M$ is a smooth manifold and $\Gamma$ 
is a finite group acting smoothly and ef{}fectively on $M$. Then, the definition of orbifold forms implies that
any $\Gamma$-invariant differential $k$-form $\alpha$ on $M$ defines an
orbifold $k$-form $\widehat{\alpha}$ on $X$, and vice-versa. Moreover, it is straightforward to
check that the exterior derivative on $M$ preserves $\Gamma$-invariance. Thus, if
$\big(\Omega^{k}(M)\big)^{\Gamma}$ denotes the space of the $\Gamma$-invariant differential $k$-forms on $M$,
and ${H^k(M)}^{\Gamma}\subset H^k(M)$ is the subspace of the de Rham cohomology classes of degree
$k$ on $M$ such that each of these classes has a representative that is a $\Gamma$-invariant differential $k$-form, then we have 
\begin{equation}\label{orbi-fcoh}
\Omega_{orb}^{k}(X)\,=\,\big(\Omega^{k}(M)\big)^{\Gamma}, \qquad  H^k(X)\,=\,{H^k(M)}^{\Gamma}.
\end{equation}
\end{remark}

\begin{definition}\label{def:symplectic_orbifold}
Let $X$ be a $7$-dimensional orbifold. We call $\varphi \in \Omega_{orb}^{3}(X)$
an {\em orbifold} $\Gtwo$ {\em form} on $X$ if, for each $p\in X$, $\varphi_{p}$ is a $\Gtwo$ form 
(in the sense of Definition \ref{def:G2-mfd}) on the 
open $\widetilde U_p\subset  \R^7$ of the orbifold chart 
$(U_p, {\widetilde U}_p, \Gamma_p, f_p)$. If in addition $\varphi$ is also closed $(d\varphi=0)$ 
we call $\varphi$ an {\em  closed orbifold} $\Gtwo$ {\em form}.
\end{definition}

 An orbifold $\Gtwo$-structure can also be defined as a reduction of the orbifold frame
bundle from $\GL(7,\RR)$ to $\Gtwo$, as in the case of smooth manifolds.

If $M$ is a smooth $7$-manifold with a closed $\Gtwo$ form $\varphi$, and $\Gamma$ is a f{}inite group acting ef{}fectively
on $M$ and preserving $\varphi$, then $\varphi$ induces an orbifold closed $\Gtwo$ form on the $7$-orbifold $\widehat{M}=M/\Gamma$.

\begin{definition} \label{def:resolution_G2}
Let $X$ be a $7$-dimensional orbifold with an orbifold closed $\Gtwo$ form $\varphi$. A {\em closed} $\Gtwo$ {\em resolution} of 
$(X,\varphi)$ consists of a smooth manifold $\widetilde{X}$ with a closed $\Gtwo$ form $\widetilde\varphi$
 and a map $\pi\colon\widetilde{X}\to X$ such that:
  \begin{itemize}
  \item $\pi$ is a dif{}feomorphism $\widetilde{X}- E \to X-S$,
  where $S\subset X$ is the singular locus and $E=\pi^{-1}(S)$ is
  the {\it exceptional locus}. Also, $\widetilde{X}- E$ is open and
    dense in~$\widetilde{X}$.
  \item $\widetilde\varphi$ and $\pi^*\varphi$ agree in the complement of a small neighbourhood of $E$.
  \end{itemize}
\end{definition}

%%%%%%%%%%%%%%%%%%%%%%%%%%%%%%%%%%%%%%%%%%%%%%%%%%%%%%%%%%%%%%%%%%%%%%%%%%%%
%%%%%%%%%%%%%%%%%%%%%%%%%%%%%%%%%%%%%%%%%%%%%%%%%%%%%%%%%%%%%%%%%%%%%%%%%%%%
%%%%%%%%%%%%%%%%%%%%%%%%%%%%%%%%%%%%%%%%%%%%%%%%%%%%%%%%%%%%%%%%%%%%%%%%%%%%
\section{Formality of manifolds and orbifolds}\label{sec:formality}
%%%%%%%%%%%%%%%%%%%%%%%%%%%%%%%%%%%%%%%%%%%%%%%%%%%%%%%%%%%%%%%%%%%%%%%%%%%%
%%%%%%%%%%%%%%%%%%%%%%%%%%%%%%%%%%%%%%%%%%%%%%%%%%%%%%%%%%%%%%%%%%%%%%%%%%%%

In this section we review some definitions and results about formal manifolds
and formal orbifolds (see \cite{BBFMT, DGMS, FHT, FernandezMunoz} for more details).

We work with {\em differential graded commutative algebras}, or DGAs,
over the field $\mathbb R$ of real numbers. The degree of an
element $a$ of a DGA is denoted by $|a|$. A DGA $(\mathcal{A},\,d)$ is said to be {\it minimal\/} if:
\begin{enumerate}
 \item $\mathcal{A}$ is free as an algebra, that is $\mathcal{A}$ is the free
 algebra $\bigwedge V$ over a graded vector space $V\,=\,\bigoplus_i V^i$, and

 \item there is a collection of generators $\{a_\tau\}_{\tau\in I}$
indexed by some well ordered set $I$, such that
 $|a_\mu|\,\leq\, |a_\tau|$ if $\mu \,< \,\tau$, and each 
 $d a_\tau$ is expressed in terms of the previous $a_\mu$, $\mu\,<\,\tau$.
 This implies that $da_\tau$ does not have a linear part.
 \end{enumerate}

Morphisms between DGAs are required to preserve the degree and to commute with the 
differential.  In our context, the main example of DGA is the de Rham complex $(\Omega^*(M),\,d)$
of a smooth manifold $M$, where $d$ is the exterior differential.

The cohomology of a differential graded commutative algebra $(\mathcal{A},\,d)$ is
denoted by $H^*(\mathcal{A})$. This space is
naturally a DGA with the product inherited from that on $\mathcal{A}$ while the differential on
$H^*(\mathcal{A})$ is identically zero. A DGA $(\mathcal{A},\,d)$ is connected if 
$H^0(\mathcal{A})\,=\,\RR$, and it is $1$-connected if, in 
addition, $H^1(\mathcal{A})\,=\,0$.

We say that $(\bigwedge V,\,d)$ is a {\it minimal model} of a 
differential graded commutative algebra $(\mathcal{A},\,d)$ if $(\bigwedge V,\,d)$ 
is minimal and there exists a morphism of differential graded algebras 
  $$
\psi \,\colon\, {(\bigwedge V,\,d)}\,\longrightarrow\, {(\mathcal{A},\,d)}
 $$ 
inducing an isomorphism $\psi^*\,\colon\, H^*(\bigwedge 
V)\,\stackrel{\sim}{\longrightarrow}\, H^*(\mathcal{A})$ on cohomology. 
In~\cite{Halperin}, Halperin proved that any connected differential graded algebra 
$(\mathcal{A},\,d)$ has a minimal model unique up to isomorphism.
For $1$-connected differential algebras, a similar result was proved by Deligne, Griffiths,
Morgan and Sullivan~\cite{DGMS,GM,Su}.

A {\it minimal model\/} of a connected smooth manifold $M$
is a minimal model $(\bigwedge V,\,d)$ for the de Rham complex
$(\Omega^*(M),\,d)$ of differential forms on $M$. If $M$ is a simply
connected manifold, then the dual of the real homotopy vector
space $\pi_i(M)\otimes \RR$ is isomorphic to the space $V^i$ of generators in degree $i$, for any $i$.
The latter also happens when $i\,>\,1$ and $M$ is nilpotent, that
is, the fundamental group $\pi_1(M)$ is nilpotent and its action
on $\pi_j(M)$ is nilpotent for all $j\,>\,1$ (see~\cite{DGMS}).

We say that a DGA $(\mathcal{A},\,d)$ is a {\it model} of a manifold $M$
if $(\mathcal{A},\,d)$ and $M$ have the same minimal model. In this case, if $(\bigwedge V,\,d)$ is the minimal
model of $M$, we have 
$$
 (\mathcal{A},\,d)\, \stackrel{\nu}\longleftarrow\, {(\bigwedge V,\, d)}\, \stackrel\psi\longrightarrow\, (\Omega^{*}(M),\,d),
$$
where $\psi$ and $\nu$ induce isomorphisms on cohomology, i.e.\ these are quasi-isomorphisms.

A minimal algebra $(\bigwedge V,\,d)$ is {\it formal} if there exists a
morphism of differential algebras $\psi\,\colon\, {(\bigwedge V,\,d)}\,\longrightarrow\,
(H^*(\bigwedge V),\,0)$ inducing the identity map on cohomology.
A DGA $(\mathcal{A},d)$ is formal if its minimal model is formal.
A smooth manifold $M$ is {formal} if its minimal model is
formal. Many examples of formal manifolds are known: spheres, projective
spaces, compact Lie groups, symmetric spaces, flag manifolds,
and compact K\"ahler manifolds. 

The formality property of a minimal algebra is characterized as follows.

\begin{theorem}[\cite{DGMS}]\label{prop:criterio1}
A minimal algebra $(\bigwedge V,\,d)$ is formal if and only if the space $V$
can be decomposed into a direct sum $V\,=\, C\oplus N$ with $d(C) \,=\, 0$,
$d$ is injective on $N$ and such that every closed element in the ideal
$I(N)$ generated by $N$ in $\bigwedge V$ is exact.
\end{theorem}

This characterization of formality can be weakened using the concept of
$s$-formality introduced in \cite{FernandezMunoz}.

\begin{definition}\label{def:primera}
A minimal algebra $(\bigwedge V,\,d)$ is $s$-formal ($s > 0$) if for each $i\,\leq\, s$
the space $V^i$ of generators of degree $i$ decomposes as a direct sum $V^i\,=\,C^i\oplus N^i$, 
where the spaces $C^i$ and $N^i$ satisfy the following conditions:
\begin{enumerate}

\item $d(C^i) = 0$,

\item the differential map $d\,\colon\, N^i\,\longrightarrow\, \bigwedge V$ is
injective, and

\item any closed element in the ideal
$I_s=I(\bigoplus\limits_{i\leq s} N^i)$, generated by the space
$\bigoplus\limits_{i\leq s} N^i$ in the free algebra $\bigwedge
(\bigoplus\limits_{i\leq s} V^i)$, is exact in $\bigwedge V$.
\end{enumerate}
\end{definition}

A smooth manifold $M$ is $s$-formal if its minimal model
is $s$-formal. Clearly, if $M$ is formal then $M$ is $s$-formal for every $s\,>\,0$.
The main result of \cite{FernandezMunoz} shows that sometimes the weaker
condition of $s$-formality implies formality.

\begin{theorem}[\cite{FernandezMunoz}] \label{fm2:criterio2}
Let $M$ be a connected and orientable compact differentiable
manifold of dimension $2n$ or $(2n-1)$. Then $M$ is formal if and
only if it is $(n-1)$-formal.
\end{theorem}

One can check that any simply connected compact  manifold 
is $2$-formal. Therefore, Theorem \ref{fm2:criterio2} implies that any
simply connected compact
manifold of dimension at most six
is formal. (This result was proved earlier in~\cite{N-Miller}.)

Note that Crowley and Nordstr\"om in \cite{CN} have introduced the {\em Bianchi--Massey tensor} on a manifold $M$, and 
they prove that if $M$ is a closed $(n-1)$-connected $(4n-1)$-manifold, with $n\,\geq \,2$, then $M$ is formal if and only if
the  Bianchi--Massey tensor vanishes.
 
For later use, we recall here the following characterization of the $s$-formality of a manifold.

\begin{lemma} [\cite{FernandezMunoz-2}]\label{fm2:criterio3}
Let $M$ be a manifold with minimal model $(\bigwedge V,d)$. Then $M$ is
$s$-formal if and only if there is a map of differential algebras
 $$
 \vartheta: (\bigwedge{V}^{\leq s}, d) \too (H^*(M),d=0),
 $$
such that the map $\vartheta^*: H^*(\bigwedge{V}^{\leq s},d) \too H^*(M)$
induced on cohomology is equal to the map $\imath^*: H^*(\bigwedge{V}^{\leq s},d) \too H^*(\bigwedge{V}, d)=H^*(M)$ induced 
by the inclusion $\imath: (\bigwedge{V}^{\leq s},d) \too (\bigwedge{V},d)$.

In particular, $\vartheta^*: H^i(\bigwedge V^{\leq s}) \too H^i(M)$ is an
isomorphism for $i\leq s$, and a monomorphism for $i=s+1$.
\end{lemma}

\begin{definition}\label{def:orbifolds-minimalmodel}
Let $X$ be an orbifold. A {\it minimal model} for $X$ is a minimal
 model $(\bigwedge V,\,d)$ for the DGA $(\Omega_{orb}^*(X),d)$. The orbifold $X$ is
 {\em formal} if its minimal model is formal.
\end{definition}

For a simply connected orbifold $X$, the dual of the real homotopy vector
space $\pi_i(X)\otimes \RR$ is isomorphic to the space $V^i$ of generators in degree $i$, for any $i$, where
$\pi_i(X)$ is the homotopy group of order $i$ of the underlying topological space in $X$.
In fact, the proof given in \cite{DGMS} for simply connected manifolds, also works 
for simply connected orbifolds (that is, orbifolds for which the topological space $X$ is simply connected).

Moreover, the proof of Theorem \ref{fm2:criterio2} given in \cite{FernandezMunoz}
only uses that the cohomology $H^*(M)$ is a Poincar\'e duality algebra. By \cite{S1}, we know that
the singular cohomology of an orbifold also satisfies a Poincar\'e duality.
Thus, Theorem \ref{fm2:criterio2} also holds 
for compact connected orientable orbifolds. That is, we have

\begin{proposition} \label{fm2:criterio2- orbif}
Let $X$ be a connected and orientable compact orbifold of dimension $2n$ or $(2n-1)$. Then 
$X$ is formal if and only if it is $(n-1)$-formal. In particular, any simply
connected compact orientable orbifold of 
dimension at most $6$ is formal.
\end{proposition}

%%%%%%%%%%%%%%%%%%%%%%%%%%%%%%%%%%%%%%%%%%%%%%%%%%%%%%%%%%%%%%%%%%%%%%%%%%%%
%%%%%%%%%%%%%%%%%%%%%%%%%%%%%%%%%%%%%%%%%%%%%%%%%%%%%%%%%%%%%%%%%%%%%%%%%%%%
%%%%%%%%%%%%%%%%%%%%%%%%%%%%%%%%%%%%%%%%%%%%%%%%%%%%%%%%%%%%%%%%%%%%%%%%%%%%
\section{A $7$-orbifold with  an orbifold closed $\Gtwo$ form} \label{sec:closedG2-orbifold}
%%%%%%%%%%%%%%%%%%%%%%%%%%%%%%%%%%%%%%%%%%%%%%%%%%%%%%%%%%%%%%%%%%%%%%%%%%%%
%%%%%%%%%%%%%%%%%%%%%%%%%%%%%%%%%%%%%%%%%%%%%%%%%%%%%%%%%%%%%%%%%%%%%%%%%%%%

Let $G$ be the connected nilpotent Lie group of dimension 7 consisting of  real matrices of the form
\begin{equation*}\label{eqn:a}
a\,=\, \begin{pmatrix} A_{1} & 0\\  0 & A_{2} \end{pmatrix},
\end{equation*}
where $A_{1}$ and $A_{2}$ are the matrices 
$$
A_{1} = \begin{pmatrix} 1&-x_2&x_1&x_4&-x_1 x_2&x_6\\ 0&1&0&-x_1&x_1&\frac{1}{2} x_{1}^2\\ 0&0&1&0&-x_2&- x_4\\ 
0& 0&0&1&0&0\\0&0& 0&0&1&x_1\\0&0& 0&0&0&1\end{pmatrix}, \quad 
A_{2} = \begin{pmatrix} 1&-x_3&x_1&x_5&-x_1 x_3&x_7\\ 0&1&0&-x_1&x_1&\frac{1}{2} x_{1}^2\\ 
0&0&1&0&-x_3&- x_5\\ 0& 0&0&1&0&0\\0&0& 0&0&1&x_1\\0&0& 0&0&0&1\end{pmatrix},
$$
where $x_{i} \in {\mathbb{R}}$, for any $i \in\{1, \cdots, 7\}$. Then, a global system of coordinate functions
$\{x_{1}, \cdots, x_{7}\}$ for $G$ is given by $x_{i}(a)=x_{i}$, with $i \in\{1, \cdots, 7\}$. 
Note that if a matrix $A\in G$ has coordinates $a_i$, then the change of coordinates
 of $a\in G$ by the left translation $L_{A}$ are given by
\begin{align*}
&L_{A}^*(x_i) = x_{i} \circ L_{A} = x_i+ a_i, \,\,\quad  i= 1,2, 3, \\ 
&L_{A}^*(x_4) = x_4+ a_{2} x_1 +  a_{4}, \\
 & L_{A}^*(x_5) = x_5+ a_{3} x_1 +  a_{5}, \\
&L_{A}^*(x_6) = x_6 - \frac{1}{2} a_{2} x_{1}^2 - a_{1} x_4 - a_{1} a_{2} x_1 + a_{6},\\
&L_{A}^*(x_7) = x_7 - \frac{1}{2} a_{3} x_{1}^2 - a_{1} x_5 - a_{1} a_{3} x_1 + a_{7}.
\end{align*}
A standard calculation shows that a basis for the left invariant $1$--forms on $G$ consists of
    \begin{equation}\label{eqn:Lucia}
    \{dx_{1}, dx_{2}, dx_{3}, dx_{4} - x_{2} dx_{1}, dx_{5} - x_{3} dx_{1}, dx_{6} + x_1 dx_{4}, dx_{7} + x_1 dx_{5}\}.
    \end{equation}
Let $\Gamma$ be the discrete subgroup of $G$ consisting of
matrices whose entries $(x_1,x_2,\ldots, x_7) \in 2\ZZ \x \ZZ^6$, that is
$x_i$ are integers and $x_1$ is even. 
It is easy to see that $\Gamma$ is a subgroup of $G$. 
So the quotient space of right cosets
\begin{equation} \label{def:nilmfd}
M\,=\,\Gamma{\backslash} G
\end{equation}
is a compact $7$-manifold. Hence the forms  
$dx_{1}, dx_{2}, dx_{3}, dx_{4} - x_{2} dx_{1}, dx_{5} - x_{3} dx_{1}, dx_{6} + x_1 dx_{4}, dx_{7} + x_1 dx_{5}$
 descend to $1$--forms $e^1, e^2, e^3, e^4, e^5, e^6, e^7$ on $M$ such that
\begin{equation}\label{eqn:struct-eq}
 d e^{i}=0, \,\, i= 1,2, 3, \qquad  d e^{4}= e^{12}, \qquad  d e^{5}= e^{13}, \qquad  d e^{6}= e^{14},  \qquad  d e^{7}= e^{15}, 
\end{equation}
and such that at each point of $M$, $\{e^{1}, e^2, e^3, e^4, e^5, e^6, e^7\}$ is a basis for the $1$--forms on $M$.
Here,  $e^{12}$ stands for $e^1 \wedge e^2$, and so on.

\begin{lemma} \label{lem:mapping-torus}
 The nilmanifold $M$ defined by \eqref{def:nilmfd} is diffeomorphic to the mapping torus  $M_{\nu}$ of the  diffeomorphism of the $6$-torus 
$\nu:T^6 = \RR^6 /\ZZ^6  \to T^6= \RR^6 /\ZZ^6$,   induced  by the linear automorphism  of $\R^6$ associated to the  matrix
 $$
\begin{pmatrix} 1 & 0 & 0 & 0 & 0 & 0 \\-2 & 1 & 0 & 0 & 0 & 0 \\2 & -2 & 1 & 0 & 0 & 0 \\
0 & 0 & 0 & 1 & 0 & 0 \\0 & 0 & 0 & -2 & 1 & 0 \\0 & 0 & 0 & 2 & -2 & 1
 \end{pmatrix}.
 $$
\end{lemma}

\begin{proof}
 Consider the projection 
 \begin{eqnarray} \label{eqn: projection p}
& p \colon  M   \to  S^1=\RR/2\ZZ \nonumber\\
& [(x_1,\ldots, x_7)]  \mapsto (x_1 + 2\ZZ)\,.
\end{eqnarray}
The fiber over $x_1 + 2 \Z \in S^1$ is the  set of equivalence classes of $\RR^6$  by the equivalence relation
 $$
 (x_2, \ldots,  x_7) \sim 
 (x_2+a_2,x_3+a_3,x_4+a_2x_1+a_4, x_5+a_3x_1+a_5, x_6-\frac12 a_2x_1^2+a_6,  x_7-\frac12 a_3x_1^2+a_7),
 $$ 
where $a_i\in\mathbb{Z}$, for $i=2,\ldots,7$.
The quotient $\R^6 / \sim $  is  then  the  $6$-torus $\RR^6 /\Lambda(x_1)$  with lattice $\Lambda(x_1) \subset \RR^6$ given by the span over $\ZZ$ of the columns of the
matrix
 $$
B(x_1) =\begin{pmatrix} 1 & 0 & 0 & 0 &0 & 0 \\ 0 & 1 & 0 & 0 &0 & 0 \\x_1 & 0 & 1 & 0 &0 & 0 \\
 0 & x_1 & 0 & 1 &0 & 0 \\ -\frac12 x_1^2& 0 & 0 & 0 &1 & 0 \\0 & -\frac12 x_1^2 & 0 & 0 &0 & 1 \end{pmatrix}.
 $$
The fiber $p^{-1}(x_1 + 2 \ZZ)=\RR^6/\Lambda(x_1)$ can be identified with the standard torus  $T^6 = \RR^6/\ZZ^6$,  by the diffeomorphism
 \begin{align*}
  f_{x_1}\colon \RR^6 /\Lambda(x_1)  & \to \RR^6/\ZZ^6 \\
[(x_2, \ldots, x_7)]  & \mapsto 
[B(x_1)^{-1}  (x_2, \ldots, x_7)].  
 \end{align*} 
Therefore, $p^{-1}([0,2]/2\ZZ) \cong ([0,2]\x T^6)/\nu$, for an
appropriate diffeomorphism $\nu: \{0\}\x T^6 \cong \{2\}\x T^6$, that
we describe next.

The manifold $M$ is obtained by identifying the two presentations
$\{0\}\times T^6$ and $\{2\}\times T^6$ of the fiber over $0+2\Z=2+2\Z$
via the map
\begin{align*}
h \colon  p^{-1}(0+2 \Z) & \to   p^{-1}(2 + 2 \Z),\\
[(x_2,x_3,x_4,x_5,x_6, x_7)]  \in {\mathbb{R}}^6 /{\Lambda}(0) & \mapsto  
[(x_2,x_3,x_4,x_5,x_6- 2 x_4, x_7-2x_5)] \in  {\mathbb{R}}^6 /{\Lambda}(2),
\end{align*}
that corresponds to the matrix
 $$
C =\begin{pmatrix} 1 & 0 & 0 & 0 &0 & 0 \\ 0 & 1 & 0 & 0 &0 & 0 \\ 0 & 0 & 1 & 0 &0 & 0 \\
 0 & 0 & 0 & 1 &0 & 0 \\ 0 & 0 & -2 & 0 &1 & 0 \\0 & 0 & 0 & -2 &0 & 1 \end{pmatrix}.
 $$
Thus $M$ is the manifold obtained from  $[0, 2] \times T^6$ by identifying the ends $\{0\}\x T^6 \cong \{2\}\x T^6$ by 
the  diffeomorphism  $\nu$ of $T^6$ induced by the linear automorphism of  $\R^6$ 
$$
(x_2, \ldots, x_7) \to E \,  (x_2, \ldots, x_7) ^T,
 $$
 where
 $$
 E= B(2)^{-1} C =\begin{pmatrix} 1 & 0 & 0 & 0 &0 & 0 \\ 0 & 1 & 0 & 0 &0 & 0 \\ -2 & 0 & 1 & 0 &0 & 0 \\
 0 & -2 & 0 & 1 &0 & 0 \\ 2 & 0 & -2 & 0 &1 & 0 \\0 & 2 & 0 & -2 &0 & 1 \end{pmatrix}.
 $$
Swapping  the coordinates $(x_2, \ldots, x_7)$ to the order $(x_2,x_4,x_6,x_3,x_5,x_7)$, we get the matrix in the statement.
\end{proof}

Now we consider the action of the finite group ${\mathbb{Z}}_2$ on $G$ given by 
\begin{eqnarray} \label{eqn:rho}
 \rho\colon G & \to & G \nonumber\\
 (x_{1}, x_{2}, x_{3}, x_{4}, x_{5}, x_{6}, x_{7}) &\mapsto &  (- x_{1}, - x_{2}, x_{3}, x_{4}, - x_{5},  - x_{6}, x_{7}),
\end{eqnarray}
where $\rho$ is the generator of ${\mathbb{Z}}_2$. 
This action satisfies  the condition  $\rho(a\cdot b)=\rho(a)\cdot \rho(b)$,
for $a, b \in G$, where $\cdot$ denotes the natural group
structure of $G$. This follows since $\rho$ is conjugation by the matrix
 $$
j\,=\, \begin{pmatrix} J_{1} & 0\\  0 & J_{2} \end{pmatrix}, \quad J_1=\mathrm{diag}(1,-1,-1,1,1,-1), \quad
J_2=\mathrm{diag}(1,1,-1,-1,-1,1),
 $$
i.e. $\rho(a)=j\,a\,j^{-1}$. 
Moreover, $\rho(\Gamma)=\Gamma$. Thus, $\rho$
induces an action on the quotient $M=\Gamma\backslash G$. Denote by 
 $$
 \rho\colon M \to M
 $$ 
the ${\mathbb{Z}}_2$-action. 
Then, the  induced action on the $1$-forms $e^i$ is given by
\begin{equation}\label{action-forms}
 \rho^*e^i= - e^i,  \,\,  i = 1, 2, 5, 6, \qquad  \rho^*e^j= e^j,  \,\,  j= 3, 4, 7.
\end{equation} 

\begin{proposition}\label{prop:7-orbifold}
The quotient space $\widehat{M}=M/{\mathbb{Z}}_2$ is a compact $7$-orbifold  with first Betti number $b_1(\widehat{M})=1$,
 and with an orbifold closed $\Gtwo$ form.
\end{proposition}

\begin{proof}
 Since the $\ZZ_2$-action on $M$ is smooth and effective, the quotient space $\widehat{M}\,=\,M/\ZZ_{2}$ is a $7$-orbifold, which
is compact because $M$ is compact. Moreover, using Nomizu's theorem \cite{Nomizu}, from  \eqref{eqn:struct-eq}
we have that the first de Rham cohomology group of $M$ is
$H^1(M)= \la [e^1], [e^2], [e^3]\ra$. Then, as a consequence 
of \eqref{orbi-fcoh} and from \eqref{action-forms}, the first 
cohomology group of $\widehat M$ is 
 $$
 H^1(\widehat M)= H^1(M)^{\ZZ_2}=
 \la [e^1], [e^2], [e^3]\ra^{\ZZ_2}= \la [e^3]\ra.
 $$
 So the first Betti number of $\widehat M$ is $b_1=1$. 

We define the $3$-form $\varphi$ on $M$ given by
 \begin{equation}\label{eqn:closed-nilmfd}
\varphi= e^{123} +  e^{145} + e^{167} - e^{246} + e^{257} + e^{347} + e^{356}.
 \end{equation}
It is clear that $\varphi$ is a $\mathbb{Z}_2$-invariant $\Gtwo$
form on $M$. We claim that $\varphi$ is also closed.
Indeed, on the right-hand side of \eqref{eqn:closed-nilmfd} all the terms, except the last 3 terms, are closed.
But $d(e^{257} + e^{347} + e^{356}) = 0$ by the
equation~\eqref{eqn:struct-eq}.
Thus $\varphi$ induces an orbifold closed $\Gtwo$ form $\widehat\varphi$ on $\widehat{M}$. 
\end{proof}

\begin{remark}
As we shall see, in Proposition~\ref{prop:deformations 3-folds} below, the
fibers $p^{-1}(x)$ of the map defined by Lemma~\ref{lem:mapping-torus}
are complex tori $\CC^3/\Lambda(x_1)$, moreover these have a
Calabi--Yau structure induced by~$\varphi$.
\end{remark}

Denote by 
$$
 \widehat{\pi} \colon M \to \widehat M 
$$
the natural projection. The singular locus $S$ of $\widehat{M}$ 
is the image by $\widehat{\pi}$ of the set $S'$ of points in $M$ that are fixed by the ${\mathbb{Z}}_2$-action 
defined by \eqref{eqn:rho}. So $S$ consists of all the $3$-dimensional spaces
 $S_{\mathbf{a}} \,=\,\widehat{\pi}(S_{\mathbf{a}}')=S_{\mathbf{a}}'/{\mathbb{Z}}_2$, where
 \begin{equation}\label{eqn:jjj2}
 S_{\mathbf{a}} '\,=
\begin{cases}
\{\Gamma\cdot (a_1, a_2, x_3, x_4, a_5, a_6,  x_7)\, \vert \, x_3, x_4, x_7
\in \R \} \subset M, & \text{ if }a_1=0\\
\{\Gamma\cdot (a_1, a_2, x_3, x_4, a_5, \frac32 a_2 + a_6 - x_4,  x_7)\, \vert \, x_3, x_4, x_7
\in \R \} \subset M, & \text{ if }a_1=1\\
\end{cases}
 \end{equation}
 and
  $$
  \mathbf{a}=(a_1, a_2, a_5, a_6) \in \mathbb{A}=\{0,1\} \x (\{0, 1/2\})^3.
 $$ 
Therefore, there are $2^4=16$ components of the singular locus of the orbifold.

The set  $S_{\mathbf{a}}'$ is included in the fiber $p^{-1}(0 + 2 \Z)$
or $p^{-1}(1 + 2 \Z)$ of the projection $p$ defined by \eqref{eqn: projection p}.
For $\mathbf{a}=(0,0,0,0)$,
$S_{\mathbf{a}}'$ is a Lie subgroup of $T^6$, hence it is abelian and so
isomorphic to a 3-torus $T^3$.
As we shall see in the next section,
$S$ is a disjoint union of $16$ copies of $T^3$.

%%%%%%%%%%%%%%%%%%%%%%%%%%%%%%%%%%%%%%%%%%%%%%%%%%%%%%%%%%%%%%%%%%%%%%%%%%
%%%%%%%%%%%%%%%%%%%%%%%%%%%%%%%%%%%%%%%%%%%%%%%%%%%%%%%%%%%%%%%%%%%%%%%%%%
%%%%%%%%%%%%%%%%%%%%%%%%%%%%%%%%%%%%%%%%%%%%%%%%%%%%%%%%%%%%%%%%%%%%%%%%%%
\section{Local model around the singular locus} \label{sec:local-model}
%%%%%%%%%%%%%%%%%%%%%%%%%%%%%%%%%%%%%%%%%%%%%%%%%%%%%%%%%%%%%%%%%%%%%%%%%%
%%%%%%%%%%%%%%%%%%%%%%%%%%%%%%%%%%%%%%%%%%%%%%%%%%%%%%%%%%%%%%%%%%%%%%%%%%

 To desingularize the orbifold $\widehat{M}=M/{\mathbb{Z}}_2$ considered in Proposition \ref{prop:7-orbifold}, 
we study here each of the $16$ connected components $S_{\mathbf{a}}$ (defined before) of the singular locus $S$ of $\widehat{M}$. 

The situation here has a partial similarity to a desingularization of
$\Gtwo$-orbifolds previously worked out by Joyce~\cite[Chap. 11,12]{Joyce2},
cf.\ also \cite{JoyceKarigiannis}.
We show that a neighbourhood of each component of the singular locus of the
orbifold $\widehat{M}$ is diffeomorphic to the product $T^3 \x (B/\ZZ_2)$ of
a 3-torus $T^3$ and the quotient $B/\pm 1$ of an open ball around zero
in~$\CC^2$, and we replace $B/\pm 1$ with an appropriate neighbourhood of the
Eguchi--Hanson space. On the other hand, the $\Gtwo$-structure near the singular
locus does not come from a flat $\Gtwo$-structure on a 7-torus and is {\em not}
torsion-free. Thus we need to modify the gluing so as to ensure a valid $\Gtwo$
form on the desingularised 7-manifold. We are able to achieve the `matching'
by exploiting the consequences of~\eqref{eqn:struct-eq} in the construction of
the $\Gtwo$-structure on the orbifold.

For each $\mathbf{a}=(a_1, a_2, a_5, a_6)\in \mathbb{A}=\{0,1\} \x (\{0, 1/2\})^3$, consider the element $a=(a_1, a_2, 0, 0, a_5, a_6, 0) \in G$. 
If $\mathbf{a}\in\mathbb{A}$ and $a_1=0$, then
\begin{equation}\label{left.tr2}
aga^{-1}\in\Gamma, \forall  g\in\Gamma.
\end{equation}
and the left translation $L_a: G \to G$ acts (diffeomorphically) on the
right cosets of $\Gamma$ (noting also that $L_a(g x)= aga^{-1} L_a(x)$).
The induced diffeomorphism $L_a: M \to M$ preserves the $\Gtwo$ form
$\varphi$ on $M$ defined by \eqref{eqn:closed-nilmfd} and satisfies
\begin{equation}\label{left.tr}
L_a(\rho(b))=a\rho(b)=\rho(a)\rho(b)=\rho(ab)=\rho(L_a(b)),
\end{equation}
for every $b \in M$.
Here we used, in the second equality in~\eqref{left.tr},
that for each $a\in\mathbb{A}$,
$a^{-1}\rho(a)=(-2a_1,-2a_2,0,2a_1 a_2, -2a_5, -2a_6+2a_1^2 a_2, 2a_1 a_5)$
which is in~$\Gamma$.
So, if $a_1 = 0$,   $L_a: M \to M$ defines an orbifold diffeomorphism $L_a:\widehat M\to \widehat M$
sending $S_{\mathbf{0}}$ to $S_{\mathbf{a}}$, where $\mathbf{0}=(0,0,0,0)\in
\mathbb{A}$.  For each $\mathbf{a}=(1,a_2,a_5,a_6)\in\mathbb{A}$, consider
$a'=(0,a_2,0,0,a_5,a_2+a_6,0) \in G$.
The corresponding orbifold diffeomorphism
$L_{a'}:\widehat M\to \widehat M$ preserving $\varphi$ is well-defined, as
above, and sends $S_{(1,0,0,0)}$ to $S_{(1,a_2,a_5,a_6)}$.
Therefore, it suffices to do the desingularization around
$S_{\mathbf{0}}$ and $S_{(1,0,0,0)}$
as we can translate it to the other singularities $S_{\mathbf{a}}$.

{}From now on, we focus on $S_{\mathbf{0}}=\{(0,0,x_3,x_4,0,0,x_7) \}\subset \widehat M$. 
The desingularization around $S_{(1,0,0,0)}$ can be obtained in a similar way
(see Remark~\ref{rmk:case.s1000}).
We consider the corresponding set
$$
S'=S_{\mathbf{0}}' =\{(0,0,x_3,x_4,0,0,x_7) \}\subset M,
$$
which is a fixed locus of the  ${\mathbb{Z}}_2$-action (given by \eqref{eqn:rho}) and is isomorphic to a 3-torus $T^3$.

 The following proposition allows us to show an appropriate local model around $S_{\mathbf{0}}$ that we will use
in the next section to desingularize $S_{\mathbf{0}}$.

 \begin{proposition}\label{prop:loc-model S'}
There exist neighbourhoods $U'$ and $U''$ of $S'$ in the manifold $M$ with $U'' \subset U'$, and there 
are closed $\Gtwo$ forms $\phi$ and $\psi$ on $M$ and $U'$, respectively which are invariant by 
the ${\mathbb{Z}}_2$-action given by \eqref{eqn:rho},  
and such that $\phi$ is equal to the $\Gtwo$ form $\varphi$, defined by \eqref{eqn:closed-nilmfd},
outside $U' \cong T^3 \x B^4_{\epsilon}$ and $\phi = \psi$ is the standard $\Gtwo$ form on  
$U''\cong T^3 \x B^4_{\epsilon/2}$
(given by~\eqref{eqn:local-G2-form} below).
\end{proposition}

 \begin{proof} 
We define a small neighbourhood $U'$ of $S'$ in $M$ as follows. 
A  point in $U'$ is given by $(x_1,\ldots, x_7)$, with $(x_1,x_2,x_5,x_6)$ small
and such that, under the equivalence relation given by the action of $\Gamma$ on the points of $U'$,
 $$
 (x_1, x_2, x_3, x_4, x_5, x_6, x_7) \sim (x_1, x_2, x_3+a_3, x_4+a_4, x_5+a_3x_1, x_6, x_7+a_7-\frac12 a_3x_1^2).
 $$
 It is natural to introduce on $U'$ the coordinates $(x_1', \ldots, x_7')$ defined by 
 \begin{equation} \label{eqn:new-coordinates}
 \begin{aligned} 
&x_5'=x_5-x_1x_3, \\
&x_7'=x_7+\frac12 x_3 x_1^2, \\
& x_j'=x_j, \quad j\neq 5,7.
 \end{aligned} 
\end{equation}
Therefore,  if $B^4_\epsilon$ is the open ball of radius $\epsilon$ in ${\mathbb{R}}^4$ centered at $0$,
$U'$ is determined by $(x_1',\ldots,x_7')$ with $(x'_1,x'_2,x'_5,x'_6) \in B^4_\epsilon$, for some  small $\epsilon>0,$ and 
 $$
 (x_1',x_2',x_3',x_4',x_5',x_6',x_7') \sim (x_1',x_2',x_3'+a_3, x_4'+a_4,x_5', x'_6,x'_7+a_7).
 $$
That is, 
\begin{equation} \label{def: U'}
U' \cong T^3 \x B^4_\epsilon\,,
\end{equation}
where $T^3=\RR^3/\ZZ^3$ has coordinates $(x_3',x_4',x_7')$, and
$B^4_\epsilon\subset \RR^4$ has coordinates $(x'_1,x'_2,x'_5,x'_6)$.
Note that for $\epsilon<\frac14$, the neighbourhoods $L_a(U') \cap L_b(U') =\emptyset$, for any 
$\mathbf{a}, \mathbf{b}\in \mathbb{A}$ distinct.

If we restrict the $1$-forms $e^1,\ldots, e^7$ to $S'$, by setting $x_1'=x_2'=x_5'=x_6'=0$,
we get 
 \begin{align*}  
 & e^5|_{S'} = dx_5-x_3dx_1=dx_5' , \\
 & e^7|_{S'} = dx_7=dx_7', \\ 
 & e^j|_{S'} = dx_j = dx_j', \quad j\neq 5,7,
 \end{align*}
since $dx_7'=dx_7+\frac12 x_1^2 dx_3+x_1x_3 dx_1$ and
$dx_5'=dx_5-x_1dx_3-x_3dx_1$. 

Thus, $e^j |_{S'} =dx_j'$, $1\leq j\leq 7$, and  the  restriction  $\varphi|_{S'}$ to $S' \subset U'$ of 
the closed $\Gtwo$ form  $\varphi$  on $M$ given by  \eqref{eqn:closed-nilmfd}, that is
 \begin{align*}  
 \varphi &= e^{123} +  e^{145} + e^{167} - e^{246} + e^{257} + e^{347} + e^{356}  \\
 &= e^{347} + e^3(e^{12}+e^{56}) -e^4(e^{15}-e^{26}) + e^7(e^{16} +e^{25})
 \end{align*}
 coincides  with the restriction  $\psi|_{S'}$ to $S'$ of  the standard $\Gtwo$ form $\psi$ on  $U' \cong T^3 \x B^4_\epsilon$ given by
 \begin{equation}\label{eqn:local-G2-form}
  \psi=
dx'_{347} + dx'_3 \wedge (dx'_{12}+dx'_{56}) -dx'_4 \wedge (dx'_{15}-dx'_{26}) + dx'_7\wedge (dx'_{16} +dx'_{25}),
 \end{equation}
that  is, we have $\psi|_{S'}=\varphi|_{S'}$. 
The notation $\psi|_{S'}$ in the previous sentence and later is used in the
sense of a section of the bundle $\Lambda^3 T^*M$ restricted to a subset
of~$M$. Here, $dx'_{12}$ stands for
$dx'_{1}\wedge dx'_{2}$, and so on, with the coordinates $x'_i$ as defined
in~\eqref{eqn:new-coordinates}. 
Moreover, using \eqref{eqn:rho} and \eqref{eqn:new-coordinates}, one can check that the $\Gtwo$ form $\psi$ on $U' \cong T^3 \x B^4_\epsilon$
is invariant by the ${\mathbb{Z}}_2$-action.

Now let us modify the $\Gtwo$-structure $\varphi$ on $M$ inside $U' \cong T^3 \x B^4_\epsilon$ so that 
it is equal to  the 3-form $\psi$ given by \eqref{eqn:local-G2-form} on a smaller neighbourhood $U''$ of $S'$.
The $3$-form  $\psi-\varphi$ is closed on $U'$, and it satisfies the condition  
 $(\psi-\varphi)|_{T^3\x \{0\}}=0$, hence it defines the zero de Rham cohomology class  on  $U'$. So 
 $\psi-\varphi=d\alpha$, for some $2$-form $\alpha$  on  $U'$.  Moreover, as 
 $|\psi-\varphi|\leq Cr$, where $r$ is the radial coordinate of $B^4_\epsilon\subset \RR^4$, we can take 
 $|\alpha|\leq Cr^2$. Indeed, following  the standard procedure 
of \cite[p. 542]{Gompf}, we can use the homotopy operator to determine $\alpha$. 
Write the 3-form $\varphi-\psi$ as  
$$
 \psi-\varphi=\beta_0\wedge dr+\beta_1.
$$  
for some forms $\beta_0$ and $\beta_1$ on $U'$ and this latter domain can be
written as $T^3\times (0,\varepsilon)\times S^3$. Now $\beta_0$ and $\beta_1$
can be regarded as forms on the 6-manifold $T^3\times S^3$ depending on the
parameter $r\in (0,\varepsilon)$. Then $d\beta_0=\partial\beta_1/\partial r$
and $d\beta_1=0$.

 The 2-form $\alpha=\int_0^r \beta_0\,dr$ 
is smooth and satisfies $d\alpha=\psi - \varphi$.

 On $B^4_\epsilon$ consider a bump function $\varrho(r)$ such that $\varrho(r)=1$ for $r\leq \epsilon/2$,
and $\varrho(r)=0$ for $\epsilon\geq r\geq 3\epsilon/4$. Define  the 3-form  $\phi$ on $M$ by
 \begin{equation}\label{eqn:modif-G2-form}
 \phi= \varphi+ d(\varrho \alpha).
  \end{equation}
Then $\phi=\varphi$ outside $U'$ and $\phi=\psi$ in 
\begin{equation} \label{def: U''}
U''\cong T^3\x B^4_{\epsilon/2}\,.
\end{equation}
Moreover, $|d\varrho|\leq C/\epsilon$ for a uniform constant, so
$|d(\varrho\alpha)| \leq C\epsilon$. For $\epsilon>0$ small enough, $\phi$ is non-degenerate,
hence it defines a  closed $\Gtwo$ form on $M$. Now, using \eqref{eqn:modif-G2-form}, one 
can check that the $\Gtwo$ form $\phi$ is still $\ZZ_2$-invariant.
\end{proof}

\begin{remark}\label{rmk:case.s1000}
In the case when $S'=S_{(1,0,0,0)}$, a neighbourhood $U'$ of $S'$ is given by
$(x_1,\ldots,x_7)$, with $x_1-1,x_2,x_5,x_6+x_4$ small and with the
equivalence relation
$$
(x_1, x_2, x_3, x_4, x_5, x_6, x_7) \sim (x_1, x_2, x_3+a_3, x_4+a_4, x_5+a_3(x_1-1), x_6-a_4, x_7+a_7-\frac12 a_3x_1^2)
$$
defined by the action of the subgroup $\{(0,0,a_3,a_4,-a_3,-a_4,a_7)\in\Gamma\,
|\, a_3,a_4,a_7\in\Z\}$ on~$U'$. In place of~\eqref{eqn:new-coordinates} we use
the following change of coordinates
\begin{alignat*}{3}
&x_1'=x_1-1,\qquad      &&x_5'=x_5-(x_1-1)x_3,  \\
&x_6'=x_6+x_4,\qquad    &&x_7'=x_7+x_5+\frac12 x_3 (x_1-1)^2,\qquad
&&x_j'=x_j, \quad j = 2,3,4.
\end{alignat*}
so \eqref{def: U'} holds with
$$
(x_1',x_2',x_3',x_4',x_5',x_6',x_7') \sim (x_1',x_2',x_3'+a_3, x_4'+a_4,x_5', x'_6,x'_7+a_7-a_3/2)
$$
and $e^j |_{S'} =dx_j'$.
\end{remark}

As a consequence of Proposition \ref{prop:loc-model S'} we have the following corollary.
\begin{corollary} \label{corollary:loc-model S0}
There exist neighbourhoods $U$ and $V$ of $S_{\mathbf{0}}$ in the orbifold $\widehat{M}=M/{\mathbb{Z}}_2$ with $V \subset U$, and there 
are orbifold closed $\Gtwo$ forms $\widehat{\phi}$ and $\widehat{\psi}$ on $\widehat{M}=M/{\mathbb{Z}}_2$ and 
$U$, respectively such that $\widehat{\phi} = \widehat{\varphi}$ outside $U,$ and
$\widehat{\phi} = \widehat{\psi}$  in the neighbourhood $V$ of $S_{\mathbf{0}}$.
Moreover, the singular locus $S$ of $\widehat{M}$ is covered by the disjoint union $\bigsqcup_{\mathbf{a}\in \mathbb{A}} L_{\mathbf{a}}(U)$.
\end{corollary}

\begin{proof}
We define the neighbourhoods $U$ and $V$ of $S_{\mathbf{0}}$ by
\begin{equation}\label{def: U-V}
 U=U'/\ZZ_2 \cong T^3 \x \big(B^4_\epsilon/\ZZ_2\big), \qquad  V=U''/\ZZ_2 \cong T^3 \x \big(B^4_{\epsilon/2}/\ZZ_2\big),
\end{equation}
where $U'$ and $U''$ are given by \eqref{def: U'} and \eqref{def: U''}, respectively. 
Consider the closed $\Gtwo$ forms $\psi$ and $\phi$ defined by \eqref{eqn:local-G2-form} and \eqref{eqn:modif-G2-form}, respectively.
By Proposition \ref{prop:loc-model S'},
both these forms are ${\mathbb{Z}}_2$-invariant, and  
hence they descend to orbifold closed $\Gtwo$ forms $\widehat{\psi}$ and $\widehat{\phi}$ on $U$ and $\widehat{M}$, 
respectively and they satisfy the required conditions.

As we have noticed in the proof of Proposition \ref{prop:loc-model S'}, we have 
$L_a(U') \cap L_b(U') =\emptyset$, for any $\mathbf{a}, \mathbf{b}\in \mathbb{A}$ distinct. 
So, $S \subset \bigsqcup_{\mathbf{a}\in \mathbb{A}} L_{\mathbf{a}}(U)$.
\end{proof}

\begin{remark} \label{rem:G2}
 Note that the $\Gtwo$ form $\psi$ given by \eqref{eqn:local-G2-form} can be defined 
as the restriction to $U'$ of the $\Gtwo$ form $\Psi$ on $T^3 \x{\mathbb{C}}^2$ defined by \eqref{eqn:G2-form} (see below).
Firstly, we see that 
in the coordinates $(x_1', \ldots, x_7')$, defined  by \eqref{eqn:new-coordinates}, the action of $\ZZ_2$ on $U'$ is given by
$$
(x_1',x_2',x_5',x_6') \mapsto (-x_1',-x_2',-x_5',-x_6'),
$$ 
and is fixing $(x_3',x_4',x_7')$. Introduce now the complex coordinates 
\begin{align*}
 & z_1=x_1'+ix_2' , \\
 & z_2= x_5'+ix_6',
 \end{align*}
so that  $U' \cong T^3 \x B^4_\epsilon$, where $B^4_\epsilon\subset \CC^2$, and the
action of $\ZZ_2$ on $\CC^2$ is given by 
\begin{equation} \label{eqn:rho-C^2}
 \begin{aligned}
 \rho\colon {\mathbb{C}}^2 & \to {\mathbb{C}}^2\\
  (z_1, z_2) &\mapsto (-z_1, -z_2).
\end{aligned}
\end{equation}
The natural $\mathrm{SU}(2)$-structure on ${\mathbb{C}}^2$ is given by the
K\"ahler form  $\omega$ and the $(2,0)$-form $\Omega$ defined, respectively, by 
\begin{equation} \label{eqn:SU(2) on C^2}
 \begin{aligned} 
  & \omega= \frac{i}{2}(dz_1 \wedge d\bar{z}_1+dz_2 \wedge d\bar{z}_2) = dx_{12}'+dx_{56}', \\
 & \Omega = dz_1\wedge dz_2=(dx_{15}'-dx_{26}') + i(dx_{25}'+dx_{16}').
\end{aligned} 
  \end{equation}
The action of $\ZZ_2$ on ${\mathbb{C}}^2$ given by \eqref{eqn:rho-C^2} preserves  both these forms. 
The standard   closed $\Gtwo$-structure on $T^3 \x{\mathbb{C}}^2$
is given by
 \begin{equation}\label{eqn:G2-form}
 \Psi = dx'_{347} + dx_3' \wedge \omega - dx_4'\wedge \Re \Omega + dx_7' \wedge \Im \Omega\,.
 \end{equation}
 So the restriction $\Psi|_{U'}$ to $U'$ coincides with the 3-form $\psi$ defined by \eqref{eqn:local-G2-form}.
Then, Corollary \ref{corollary:loc-model S0}
implies that
 \begin{equation}\label{eqn:G2-form-orbif}
 {\widehat{\Psi}}|_{V} = \widehat{\phi} = \widehat{\psi}
 \end{equation} 
  in the neighbourhood $V$ of $S_{\mathbf{0}}$, where $\widehat{\Psi}$ is the orbifold closed $\Gtwo$ form induced by 
$\Psi$ on  $T^3 \x({\mathbb{C}}^2/\ZZ_2)$, and ${\widehat{\Psi}}|_{V}$ is the restriction  of $\widehat{\Psi}$ to $V$.
\end{remark}

%%%%%%%%%%%%%%%%%%%%%%%%%%%%%%%%%%%%%%%%%%%%%%%%%%%%%%%%%%%%%%%%%%%%%%%%%%
%%%%%%%%%%%%%%%%%%%%%%%%%%%%%%%%%%%%%%%%%%%%%%%%%%%%%%%%%%%%%%%%%%%%%%%%%%
%%%%%%%%%%%%%%%%%%%%%%%%%%%%%%%%%%%%%%%%%%%%%%%%%%%%%%%%%%%%%%%%%%%%%%%%%%
\section{Resolving the singular locus}\label{sec:resolution}
%%%%%%%%%%%%%%%%%%%%%%%%%%%%%%%%%%%%%%%%%%%%%%%%%%%%%%%%%%%%%%%%%%%%%%%%%%
%%%%%%%%%%%%%%%%%%%%%%%%%%%%%%%%%%%%%%%%%%%%%%%%%%%%%%%%%%%%%%%%%%%%%%%%%%

 In this section we desingularize the singular locus $S$ of $\widehat{M}$ to get a smooth compact $7$-manifold 
$\widetilde{M}$ diffeomorphic to $\widehat{M}$ outside $S$, and such that $\widetilde{M}$
has the required properties, i.e.\ with first Betti number $b_1(\widetilde{M})=1$, with no torsion-free $\Gtwo$-structures and with a 
closed $\Gtwo$ form $\widetilde{\varphi}$ such that $\widetilde{\varphi}=\widehat{\varphi}$ outside 
a neighbourhood of $S$, where $\widehat{\varphi}$ is the orbifold closed $\Gtwo$ form on $\widehat{M}$
given in the proof of Proposition \ref{prop:7-orbifold}.

\begin{theorem}\label{th:resolution}
A closed $G_2$-resolution of $(\widehat{M},\widehat{\varphi})$ {\rm (}in the
sense of Definition \ref{def:resolution_G2}{\rm )} induced by the
blow up of the origin in $\CC^2$ via the local model defined by
Corollary~\ref{corollary:loc-model S0} produces
a smooth compact  manifold $\widetilde M$ with a closed $\Gtwo$ form $\widetilde\varphi$ and with 
first Betti number $b_1(\widetilde{M})=1$.
\end{theorem}

\begin{proof} 
We know that doing the desingularization 
around the component $S_{\mathbf{0}}$ of $S$, we can translate it to the other components 
$S_{\mathbf{a}}$ of the singular locus $S$ via  the diffeomorphism $L_a:\widehat M\to \widehat M$ defined in section \ref{sec:local-model}.

Let $V$ be the neighbourhood of $S_{\mathbf{0}}$ given by \eqref{def: U-V} with the orbifold 
closed $\Gtwo$ form ${\widehat{\Psi}}|_{V}$
induced on $V$ by the $\Gtwo$ form $\Psi$ defined by \eqref{eqn:G2-form}. 
In order to desingularize $S_{\mathbf{0}}$, we shall replace the factor $B^4_{\epsilon/2}/\ZZ_2$ of $V$ by a 
smooth $4$-manifold that agrees with $B^4_{\epsilon/2}/\ZZ_2$ in a neighbourhood of its boundary.

Firstly we consider the complex orbifold $X=\CC^2/\ZZ_2$. By Remark \ref{rem:G2}, we know that the action 
of $\ZZ_2$ on $\CC^2$ defined by \eqref{eqn:rho-C^2} preserves the natural integrable 
$\mathrm{SU}(2)$-structure $(\omega, \Omega)$ on ${\mathbb{C}}^2$ given by \eqref{eqn:SU(2) on C^2}. 
(Thus $\ZZ_2$ is a finite subgroup of $\SU(2)$.) We resolve the singularity of $X=\CC^2/\ZZ_2$ to get a smooth manifold 
$\widetilde{X}$ with a (non-torsion-free) $\mathrm{SU}(2)$-structure. This goes as follows: 
take the blow-up $\widetilde{\CC{}^2}$ of $\CC^2$ at the origin. This is  given by
 $$
\widetilde{\CC{}^2}=\Big\{ \big(z_1,z_2,[w_1,w_2]\big) \in \CC^2 \x\CP^1\, \big{|} \, w_1z_2=w_2z_1\Big\}.
 $$
Now we quotient $\widetilde{\CC{}^2}$  by $\ZZ_2$ in order to get a smooth manifold
$$
\widetilde{X}=\widetilde{\CC{}^2}/\ZZ_2.
$$
and a map $\pi\colon\widetilde{X}\to X$ such that $\pi$ is a diffeomorphism 
$\widetilde{X}- \pi^{-1}(0) \to X-\{0\}$.

It is known that the $(2,0)$-form $\Omega=dz_1\wedge dz_2$ on $X$ extends to a nowhere vanishing 
$(2,0)$-form on $\widetilde X$, that we call $\Omega$ again (see~\cite{dancer}).
This can be easily checked
as follows, using the two affine charts. For the first one, we  take $w_1=1$, $w_2=w$,
$z_1=z$, $z_2=w z$, so the chart of $\widetilde{\CC{}^2}$ is parameterized by $(z,w)\in \CC^2$.
The quotient  by $\Z_2$ is given by $(z,w)\mapsto (-z,w)$, so $\widetilde X$ is parameterized
by $(u,w)\in \CC^2$, with $u=z^2$. The form $\Omega$ in these  coordinates $(u,w)$ has the following expression 
 \begin{align*}
 \Omega =& dz_1\wedge dz_2= dz \wedge d(wz)=z dz \wedge dw =\frac12 du \wedge dw.
 \end{align*}
Thus $\Omega$  is non-zero and  is defined  on  the whole chart. The computations for  other chart are similar.

We now consider a family of K\"ahler $(1,1)$-forms on $\CC^2-\{0\}$ 
that extend to $\widetilde{\CC{}^2}/\ZZ_2$. These determine the Eguchi--Hanson metric (see \cite{EG}, \cite[p.153]{Joyce2}).
Define $\omega_t= \frac{1}{2i} \partial \overline{\partial} f_t$, where $r=||(z_1,z_2)||$ and
 $$
 f_t(r)= (r^4 + t^4)^{1/2} + 2t^2 \log r - t^2 \log((r^4 + t^4)^{1/2} + t^2).
 $$
Also note that $f_t - r^2 = t^2 h(t,r)$, where 
$h(t,r)= t^2((r^4 + t^4)^{1/2} + r^2)^{-1} + 2\log r - \log((r^4 + t^4)^{1/2} + t^2)$, 
which is smooth on $\RR^2-\{0\}$. Take $\varrho$ a bump function such that $\varrho \equiv 1$ 
if $r \leq \epsilon/4$ and $\varrho \equiv 0$ if $r\geq {\epsilon}/{2}$. 
Define  $\widetilde{\omega}_t = \omega_{\CC^2} + \frac{1}{2i}\partial \overline{\partial}(\varrho( r^2 - f_t))$. 
We claim that there exists $t>0$ such that $\widetilde{\omega}_t$ is non-degenerate on $B_{\epsilon}^4$. 
In order to check it on the neck $B_{\epsilon/2}^4-B_{\epsilon/4}^4$, 
let $m>0$ be such that any $\omega$ with $||\omega - \omega_{\CC^2}|| < m$ is symplectic.
As $h$ is bounded in $C^2$ on $[0,1] \times [{\epsilon}/4,{\epsilon}/{2}]$, we can choose
$t>0$ so that  $||\partial \overline{\partial}(\varrho( r^2 - f_t))|| < m$. Then
$\widetilde{\omega}_t$ is symplectic.

We define the $\Gtwo$ form $\widetilde\Psi$ on $T^3\x \big(\widetilde{B{}^4_{\epsilon/2}}/\ZZ_2\big)$   by
 $$
 \widetilde\Psi= dx_{347}'+ dx_3' \wedge  \widetilde \omega_t  - dx_4'\wedge \Re \Omega + dx_7' \wedge \Im \Omega.
 $$
Thus,  for $(\epsilon/2)-\eta \leq r < \epsilon/2$, we have 
 $\widetilde\Psi =\widehat\Psi$ on $T^3\x \big(\widetilde{B{}^4_{r}}/\ZZ_2\big)$, and hence
 $\widetilde\Psi =\widehat\phi$ on $T^3\x \big(\widetilde{B{}^4_{r}}/\ZZ_2\big)$ by \eqref{eqn:G2-form-orbif}.
 Now we glue $T^3\x \big(\widetilde{B{}^4_{\epsilon/2}}/\ZZ_2\big)$ endowed with this $\Gtwo$ form  
 $\widetilde\Psi$ to $\widehat{M}-\Big(T^3\x \big(B{}^4_{(\epsilon/2)-\eta}/\ZZ_2\big)\Big)$ with the $\Gtwo$ form 
$\widehat{\phi}$ given in Corollary  \ref{corollary:loc-model S0}.
These two glue nicely to give a $\Gtwo$ form $\widetilde{\varphi}$
on the resulting smooth manifold $\widetilde M$. 

The map $\pi\colon\widetilde{X}\to X$ 
defines a map that we denote by the same symbol $\pi\colon\widetilde{M}\to \widehat{M}$, which satisfies the conditions
of Definition \ref{def:resolution_G2}. Thus, $(\widetilde M, \widetilde{\varphi})$ is a closed $\Gtwo$-resolution
of $(\widehat{M}, \widehat{\varphi})$.

 Finally, that $b_1(\widetilde{M})=1$ follows from Proposition \ref{prop-cohom} below.
\end{proof}

\section{Topology of the constructed manifold}

The next proposition gives some topological invariants of $\widetilde{M}$
(cf. \cite[\S 12.1]{Joyce2}). In fact, we shall be able to further determine
the de Rham cohomology ring of $\widetilde{M}$ later in this section.

\begin{proposition} \label{prop-cohom}
 There is an isomorphism 
 $$
 H^*(\widetilde M) \cong  H^*(\widehat M) \oplus \left( \bigoplus_{i=1}^{16} H^*(T^3) \ox [E_i] \right),
 $$
where $[E_i]\in H^2(\widetilde M)$ is the class of the exceptional divisor 
$E_i \subset\widetilde{X}=\widetilde{\CC^2}/\ZZ_2$ with $1 \leq i \leq 16$.
\end{proposition}

\begin{proof}
   Let $V\subset \widehat M$ be a neighbourhood of the exceptional divisors, that is $V=\bigsqcup_i  V_i$, where 
 $V_i\cong T^3 \x (B^4_\epsilon/\ZZ_2) \sim T^3$, 
where $\sim$ means homotopy equivalence.
Let $U$ be the complement $\widehat M -\bigsqcup_i  E_i$.
Then $U\cap V =\bigsqcup_i  (U\cap V_i)$, where $U\cap V_i \sim T^3  \x (S^3/\ZZ_2)$.

 Let $\widetilde V\subset \widetilde M$ be the preimage of $V$ under $\pi:\widetilde M\to \widehat M$.
Then $\widetilde V=\bigsqcup_i  \widetilde V_i$ with 
$\widetilde V_i\cong T^3 \x (\widetilde{B^4_\epsilon}/\ZZ_2) \sim T^3 \x E_i$ and  $E_i\cong \CP^1 \cong S^2$. 

The map $\pi^*:  H^k(\widehat M) \to  H^k(\widetilde M)$ is injective. In fact, let $\alpha\in H^k(\widehat M)$ be a non-zero element.
As the cohomology of $\widehat M$ is a Poincar\'e duality algebra, there is some $\beta \in H^{7-k}(\widehat M)$ such
that $\alpha\cdot \beta =[\widehat M]$. Applying $\pi^*$, and noting that $\pi:\widetilde M\to \widehat M$ is a degree $1$ map,
we have that $\pi^*\alpha\cdot \pi^*\beta =[\widetilde M]$. Then $\pi^*\alpha\neq 0$.

We write the Mayer--Vietoris sequences associated to $\widehat M=U\cup V$ and $\widetilde M=U\cup \widetilde V$ as
 $$
 \begin{array}{ccccccccc}
   \to & H^{k-1}(U\cap V) & \stackrel{\delta^{k-1}_1}{\to} & H^k(\widehat M)  & \to & H^k(U) \oplus H^k(V) & \to & H^k(U\cap V) & 
 \stackrel{\delta^{k}_1}{\to}  \\
   & || & & \downarrow\pi^*  & & \downarrow & & ||  \\
   \to & H^{k-1}(U\cap\widetilde V) & \stackrel{\delta^{k-1}_2}{\to} 
   & H^k(\widetilde M)   & \to & H^k(U) \oplus H^k(\widetilde V) & \to & H^k(U\cap \widetilde V) &  \stackrel{\delta^{k}_2}{\to}   \\
    &  & & \downarrow  & & \downarrow   \\
    &  & & Q & \stackrel{f}{\to} & \bigoplus_{i=1}^{16}   H^{k-2}(T^3) \ox [E_i]
 \end{array}
 $$
where $Q$ is the cokernel of $\pi^*$. It is clear that $\im \delta^{k-1}_1=\im \delta^{k-1}_2$. This happens for all
$k$. So $\ker \delta^k_1=\ker \delta^k_2$. Applying the snake lemma to the
vertical exact sequences in the second and third columns of the diagram above,
we have an exact sequence
$$0 \to \im \delta^{k-1}_1\to \im \delta^{k-1}_2 \to \ker f \to \ker \delta^k_1\to \ker \delta^k_2 \to \im f\to 0.$$
This concludes that $f$ is an isomorphism.
Therefore there is an exact sequence
 $$
 0 \to H^*(\widehat M) \to H^*(\widetilde M) \to \bigoplus_{i=1}^{16} H^*(T^3) \ox [E_i]\to 0\,,
 $$
where $[E_i]\in H^2(\widetilde M)$ is the class of the exceptional divisor $E_i\subset \widetilde{\CC^2}/\ZZ_2$ $(1 \leq i \leq 16)$.

Now let us construct a splitting of the above exact sequence. For this, we take the Thom form $\eta_i$ of 
each of the exceptional divisors $E_i \subset \widetilde{\CC^2}/\ZZ_2$. Let $E$ be one of these exceptional
divisors. The Thom form of $E$ is a compactly supported $2$-form $\eta$ on a neighbourhood of 
$E$  such that $[\eta]=[E]$. Moreover $\eta^2$ represents a $4$-form such that 
 $\int_F \eta^2= [F]\cdot [E]^2=-2$ for each fiber $F=\{p\}\x\widetilde{\CC^2}/\ZZ_2$ of
$\tilde V=T^3 \x \widetilde{\CC^2}/\ZZ_2$. If  $\lambda$ is the bump $4$-form on the origin of $\CC^2$,
pulled-back to $\widetilde{\CC^2}/\ZZ_2$, then $[\eta^2]=-2[\lambda]$. Pulling-back to $\tilde V=T^3 \x \widetilde{\CC^2}/\ZZ_2$,
we have that $[\eta^2]=-2 [T^3]$. With this we construct the compactly supported cohomology of 
$\widetilde V_i=T^3 \x \widetilde{\CC^2}/\ZZ_2$ as
 the forms $\bigwedge (e^3, e^4, e^7) \wedge [\eta_i]$. This gives the splitting.
\end{proof}

The algebra structure of $H^*(\widetilde M)$ can be described explicitly as follows. Under the isomorphism
 given in Proposition \ref{prop-cohom}, i.e.
 $$
 H^*(\widetilde M) \cong  H^*(\widehat M) \oplus \left( \bigoplus_{i=1}^{16} H^*(T^3) \ox [E_i] \right),
 $$
the elements of $H^*(\widehat M)$ multiply following its algebra structure. Moreover, an element $\alpha\in H^*(\widehat M)$
and $\beta\ox [E_j]$ multiply as $\alpha \cdot (\beta\ox [E_j])= (i_j^*\alpha\wedge\beta)\ox [E_j]$, where
$i_j:S_j\subset \widehat M$ is the inclusion of the $j$-th component $S_j$ of the singular locus.
Finally, 
 $$
   [E_j]\cdot [E_j]=-2 [\lambda]= -2 e^{1256},\qquad
 [E_i]\cdot [E_j] = 0 \text{ \ if } i\neq j,
 $$
since it is the Poincar\'e dual of the $T^3$ given by coordinates $(x_3,x_4,x_7)$. So
  $(\beta\ox [E_j])\cdot (\gamma \ox [E_j])= -2 \beta\wedge \gamma\wedge [\lambda]\in H^*(\widehat M)$.
In summary, 
 \begin{align*}
 (\alpha_1, \sum_j  \beta_{1j} \ox [E_j]) & \cdot (\alpha_2, \sum_j  \beta_{2j} \ox [E_j]) = \\ &=
 \left(\alpha_1 \wedge \alpha_2 -2 \sum_j  \beta_{1j}\wedge\beta_{2j}\wedge \lambda, \sum_j  (\beta_{1j}\wedge \alpha_2+
 \alpha_1\wedge \beta_{2j}) \ox [E_j]\right).
 \end{align*}
To complete the proof of Theorem \ref{th:resolution}, we compute the Betti
numbers of $\widetilde{M}$.
 Recall that according to Nomizu's theorem
\cite{Nomizu}, the de Rham cohomology of the nilmanifold $\Gamma\backslash G$
is isomorphic to Chevalley--Eilenberg cohomology of the Lie algebra of~$G$.
We easily find that
the de Rham cohomology groups $H^2 (M)$ and $H^3 (M)$ of the nilmanifold $M$ are
 \begin{align*}
H^2 (M) =&\, \la [e^{16}], [e^{17}], [e^{23}], [e^{24}], [e^{25} + e^{34}], [e^{35}], [e^{27} -e^{45} - e^{36}] \ra\,, \\
H^3 (M) =&\, \la [e^{136}], [e^{146}], [e^{147}],  [e^{157}], [e^{167}],  [e^{234}], [e^{235}], [e^{236} + e^{245}], \\
&  [e^{237} + e^{345}], [e^{246}], [e^{357}],  [e^{247}+e^{256}+e^{346}],  [e^{257}+e^{347}+e^{356}]\ra\,,
\end{align*}
and thus
\begin{align*}
 H^2(\widehat{M}) = H^2({M})^{\Z_2}  = & \, \la [e^{16}], [e^{25} + e^{34}]\ra\,,\\
 H^3(\widehat{M}) = H^3({M})^{\Z_2} = & \, \la [e^{136}],[e^{146}],  [e^{157}], [e^{167}], [e^{235}], \\
 & [e^{236} + e^{245}],   [e^{246}],   [e^{257}+e^{347}+e^{356}]\ra\,.
 \end{align*}
Then, Proposition \ref{prop:7-orbifold} and Proposition \ref{prop-cohom} imply that the Betti numbers of $\widetilde M$ are as follows:
\begin{align} \label{eqn:finishing} 
b_1(\widetilde M) =& \,  b_1(\widehat M) = 1, \nonumber \\
 b_2(\widetilde M) =& \, b_2(\widehat M)+16 = 18, \\
 b_3(\widetilde M) =& \, b_3(\widehat M)+16  \, b_1(T^3)= 56. \nonumber
 \end{align}

\begin{proposition}\label{prop:fundamental-group}
The compact manifold $\widetilde M$  has fundamental group $\pi_1(\widetilde M) = \Z.$
\end{proposition}

\begin{proof}
 Let $\widehat{\pi} :M\to \widehat M$ be the quotient map. Fix $p_0\in M$ to be the point with coordinates $(0,\ldots,0)$, and
 let $q_0 =\widehat{\pi}(p_0)$ be the image of $p_0$ under the projection $\widehat{\pi}$.
Let $\gamma_1,\ldots, \gamma_7$ be the loops on $M$, where $\gamma_i$ is the
image under $\widehat{\pi}$ of the path from $p_0$ to $e_i=(0,\ldots, \stackrel{(i)}{1},\ldots, 0)$. 
These are generators of the fundamental group 
$\pi_1(M,p_0)$ subject to the relations
 \begin{equation}\label{eqn:gamma}
 [\gamma_1, \gamma_2]=\gamma_4, [\gamma_1,\gamma_3]=\gamma_5, [\gamma_1,\gamma_4]=\gamma_6, [\gamma_1,\gamma_5]=\gamma_7 ,
 \end{equation}
and the fact that the other commutators are zero, i.e. $\gamma_2,\gamma_4$ commute, etc.
 
We claim that any loop $\bar{\alpha}$ on $\widehat M$ lifts to $M$ (non-uniquely). The (closed) portions of
$\bar{\alpha}$ that lie in the orbifold locus lift uniquely. The (open) part of $\bar{\alpha}$ that lies off the orbifold locus
lift to two possible paths (since over there $\widehat{\pi}$  is a double covering). Take any of those lifts. 
The result is a continuous path $\alpha$ on $M$ such that $\bar{\alpha}=\widehat{\pi}\circ\alpha$.
This is a well-defined loop, because the end-point lifts uniquely to the base point.
This concludes that  $\pi_{1}(\widehat M,q_0)$ is generated by the images 
 $\bar{\gamma}_i=\widehat{\pi}\circ {\gamma_i}$, $1\leq i\leq 7$.

Now recall that $\ZZ_2$ acts by (\ref{eqn:rho}). Under it, the image of $\gamma_1$ is the same as the path
from $(0,0,\ldots, 0)$ to $(\frac12,0,\ldots, 0)$ followed by the same path in the reversed direction.
This is contractible, hence $\bar\gamma_1=0$. The same happens with $\gamma_2$, so $\bar \gamma_2=0$.
Using the relations (\ref{eqn:gamma}),
we conclude that $\pi_1(\widehat M,q_0)=\la \bar\gamma_3\ra$. Therefore
$\pi_1(\widehat M)\cong \ZZ$, since $b_1 (\widehat M)=1$.

Now we prove that the resolution process does not alter the fundamental group. Let us treat the case
of the orbifold locus $S_0 \cong T^3 \subset \widehat M$. Let  $\pi:\widetilde M\to \widehat M$ be
the resolution map. Take $U$ a neighbourhood of $S_0$, and $V=\widehat{M}-S_0$. Consider 
$\widetilde U=\pi^{-1}(U)$ and $\widetilde V=\pi^{-1}(V)$.
Then by Seifert--van Kampen, $\pi_1(\widehat M)$ is the amalgamated sum of $\pi_1(U)$ and $\pi_1(V)$ over $\pi_1(U\cap V)$.
And $\pi_1(\widetilde M)$ is the amalgamated sum of $\pi_1(\widetilde U)$ and $\pi_1(\widetilde V)$ over $\pi_1(\widetilde U\cap \widetilde V)$.
 Note that $\widetilde V\cong V$,  $\widetilde U\cap \widetilde V\cong U\cap V$, and $U\sim T^3$, $\widetilde U\sim T^3\x \CP^1$, so
that $\pi_1(\widetilde U)\cong \pi_1(U)$. Therefore $\pi_1(\widetilde M)\cong \pi_1(\widehat M)\cong \ZZ$.
\end{proof}

 Next, we complete the properties of $\widetilde M$ proving that it is formal, and that it does not admit any torsion-free $\Gtwo$-structure.
 
 \begin{proposition}	\label{prop:formresolut}
The compact manifold $\widetilde M$ is formal.
\end{proposition}

\begin{proof}
We are going to first check that the orbifold $\widehat{M}$ is formal. Note 
that the cohomology group $H^3(\widehat{M})$ of $\widehat{M}$ decomposes as
$$H^3(\widehat M)= A\oplus B, $$ where
$A=\la  [e^{136}], [e^{235}] \ra$ and $B= \la  [e^{146}],  [e^{157}], [e^{167}], [e^{246}],
[e^{236} + e^{245}],  [e^{257}+e^{347}+e^{356}]\ra$. Then,  the multiplication by $[e^3]$
vanishes on $A$, and it defines an isomorphism $[e^3]:H^2(\widehat M) \to A$.
Moreover, the multiplication by $[e^3]$ is injective on $B\to H^4(\widehat M)$. For this just check that the
map $H^3(\widehat M)\x H^3(\widehat M) \to \RR$, $(\alpha,\beta)\mapsto \int \alpha\wedge \beta\wedge e^3$
has matrix (on the given basis  of $H^3(\widehat M)$) of the form $$\left(\begin{array}{cccc}
0 & \ldots & 0 & 1 \\ 0 & \ldots & 1 & * \\ \vdots & && \vdots \\ 1 & * &\ldots & *\end{array}\right).
$$
On the other hand, with respect to the basis $\alpha_1=[e^{16}]$, $\alpha_2=[e^{25}+e^{34}]$ of $H^2(\widehat M)$,
we have $\alpha_1^2=0$ and $\alpha_2^2=2[e^{2345}]= -2 [e^3]\wedge
[e^{236}+e^{245}]$, but $\alpha_1\wedge \alpha_2=2[e^{1256}]\not=0$ as
$e^{1256}-e^{1346}=de^{456}$ from~\eqref{eqn:struct-eq}.

Since $M$ is a compact nilmanifold, the minimal model of $M$ is the minimal DGA 
$(\bigwedge V,d)$, where $V=\la e^1,\ldots, e^7\ra$ and the differential $d$ is defined by \eqref{eqn:struct-eq}.
Let $F=\ZZ_2$ be the finite group acting on $M$, and on the minimal model. So $((\bigwedge V)^F,d)$ is a model 
(not minimal) of $\widehat M=M/F$. Let $\psi:(\bigwedge W, d) \to ((\bigwedge V)^F, d)$ be a minimal model of $\widehat{M}$. 
Using notation of Definition \ref{def:primera}, we write $W^{i} =
C^{i} + N^{i}$, $i\leq 3$. 
We shall write $a_j,b_j,c_j,\eta_j$ for the
generators of, respectively $C^i$, $i\leq 3$, and $N^3$.  Then,
 \begin{align*} 
W^1=& \, C^1= \la a_1\ra\,, \\
W^2=& \, C^2=\la b_1,b_2\ra\,, \\
W^3= & \, C^3 \oplus N^3, \,\, \text{ where } \,\, C^3= \la c_1,c_2,c_3,c_4,c_5,c_6\ra\,, 
N^3=\la n_1,n_2 \ra\,,
 \end{align*}
 the differential $d$ is given by $d(C^i) =0$, $d n_1=b_1^2$,  $d n_2= b_2^2+2 a_1c_5$, 
 and the morphism $\psi : (\bigwedge W, d) \to ((\bigwedge V)^F, d)$ of differential algebras is defined by
 \begin{align*} 
&\psi(a_1)=e^3, \qquad \qquad\psi(b_1)=e^{16},\qquad \qquad \psi(b_2)=e^{25}+e^{34}, \\
&\psi(c_1)= e^{146}, \,\,\qquad \quad \psi(c_2)=e^{157}, \qquad\quad \,\, \,\,  \psi(c_3)=e^{167}, \\
&\psi(c_4)=e^{246},\qquad\quad \,\,\psi(c_5)=e^{236} + e^{245}, \quad \psi(c_6)=e^{257}+e^{347}+e^{356}, \\
&\psi(\nu_j)= 0, \, j=1, 2.
 \end{align*}   
Now we can prove that $\widehat M$ is $3$-formal, and so it is formal by Proposition \ref{fm2:criterio2- orbif}.
 For this we have to look at the closed elements of $I(N^3)\subset \bigwedge W^{\leq 3}$, and check that 
the image through $\psi$ is exact. However this is clear since $\psi(N^3)=0$.

To check the formality of $\widetilde M$, now we have to work out the $3$-minimal model of it, with
the algebra structure of $H^*(\widetilde M)$ given above. Note that there is a Thom form $\eta_i$
such that $[\eta_i]=[E_i]$. It is clear that $\eta_i\wedge \eta_j=0$ for $i< j$. For $1\leq k\leq 16$ we take $3$-forms $\theta'_{k}$
such that $d\theta'_{k}=\eta_k^2-2e^{1256}$. 
As the exceptional divisors $E_k$ lie over the $3$-tori (\ref{eqn:jjj2}), 
the Thom forms $\eta_k$ can be defined only with the coordinates of the fibers, that is, 
 \begin{equation}\label{eqn:aaa3}
 L_{e_a}(\eta_k)= i_{e_a}(\eta_k) =0, \text{ for } a=3,4,7
 \end{equation}
Therefore, the same property can be arranged for  $\theta'_k$. Any $5$-form
with that property vanishes, so
 \begin{equation}\label{eqn:aaa1}
 \theta'_k\wedge \eta_i=0,
 \end{equation}
We also note that $[e^{16}]|_{E_i}=0$, hence there are $3$-forms $\theta''_i$
such that $d\theta''_i=e^{16}\wedge \eta_i$. The forms $\theta''_i$ are arranged to 
satisfy (\ref{eqn:aaa3}). Thus
 \begin{equation}\label{eqn:aaa2}
 \theta''_i\wedge \eta_j=0.
 \end{equation}

Therefore, the minimal model of $\widetilde M$ must
be a differential graded $(\bigwedge \widetilde{W}, \widetilde{d})$
where $\widetilde{W}$ is
the graded vector space $\widetilde{W}\,=\,\bigoplus_i \widetilde{W}^i$ with
 \begin{align*} 
\widetilde{W}^1 &=W^1, \\
\widetilde{W}^2&=W^2\oplus S^2,\,\,\,\,\,\,   S^2=\la B_i \, | \,1\leq i\leq 16\ra\,, \\
\widetilde{W}^3&=W^3 \oplus S^3 \oplus R^3,  \,\,\, 
  S^3=\la C_{i}^4, C_{i}^7 \, |\,1\leq i\leq 16\ra ,  \\ 
 &  \qquad \qquad \qquad \qquad  \, \, R^3 = \la D_{ij}  \, |\,  1\leq i<j\leq 16 \ra \oplus \la D_{k}' , D_k''\, |\,  1\leq k\leq 16 \ra ,
 \end{align*}
and the differential $\widetilde{d}$ is given by $\widetilde{d}|_{W^i}=d$, 
$  \widetilde{d}(B_i)=  \widetilde{d}(C_{i}^4)=\widetilde{d}(C_{i}^7)=0$,
and
  $$
  \widetilde{d}(D_{ij})=B_iB_j, \,\, \widetilde{d}(D_{k}')=B_k^2 +2(b_1b_2+a_1c_1),  \,\,  \widetilde{d}(D''_{k})=b_1 B_k \, .
  $$ 
Now, we define the map of differential algebras
$\vartheta: (\bigwedge{\widetilde{W}}^{\leq 3}, \widetilde{d}) \too  (\Omega^*(\widetilde{M}),d)$,
by $\vartheta|_W=\psi$ and 
 \begin{align*} 
  & \vartheta(B_i)=\eta_i,  \,  \vartheta(C_{i}^4)=e^4\wedge \eta_i, \,   \vartheta(C_{i}^7)=e^7\wedge \eta_i,  \, 
 \vartheta( D_{ij})= 0, \,   \vartheta(D_k')=\theta'_k\, , \vartheta(D_k'')=\theta''_k\, ,
 \end{align*}   
where $1\leq i\leq 16$ and $1\leq i<j\leq 16$ and $1\leq k\leq 16$. This is a $3$-minimal model of $\widetilde{M}$.

To check the $3$-formality, observe that  $\widetilde N^3=N^3\oplus R^3$.
We have to see that the closed elements of degree $\leq 7$ in $I(\widetilde{N})$ are exact. 
In degree $4$ there are no closed elements. In degree $5$, we have the elements
 \begin{equation}\label{eqn:aaa4}
 \begin{aligned}
& b_1D_{jk} - B_jD_k'', \quad B_i(D_j'- D_k')- (D_{ij}B_j-D_{ik}B_k).
 \end{aligned}
 \end{equation}
The image via $\vartheta$ is zero by using (\ref{eqn:aaa1}) and (\ref{eqn:aaa2}), 
so the elements are exact. In degree $6$, we only have 
$a_1(D_k' B_j + 2b_2D_j'' - B_k D_{jk})$. This lies in $H^6(\widetilde M)$,
and multiplying by $a_1$, it vanishes. By Poincar\'e duality, it defines the zero cohomology class. 
In degree $7$, the closed elements are those in (\ref{eqn:aaa4}) times $b_l$ or $B_l$, and the elements
  \begin{align*} 
 & B_jB_k (D_l'-D_m') - D_{jk}(B_l^2-B_m^2), \quad   b_1^2 D_{jk}- B_jB_k n_1, \\ 
 &  (b_2^2 + 2 a_1c_5)D_{jk} - B_jB_k n_2, \qquad \qquad B_i b_1 D_j'' - B_i B_j n_1. 
  \end{align*} 
All of them are clearly exact. 
\end{proof}

\begin{theorem} \label{th:notorfree}
The compact manifold $\widetilde M$ does not admit any torsion-free $\Gtwo$-structure.
\end{theorem}

\begin{proof}  We prove the theorem by contradiction. Suppose that $\widetilde{M}$ admits a 
torsion-free $\Gtwo$-structure with associated metric $g$. Then, the restricted holonomy group of $g$ is a subgroup of $\Gtwo$. 
By \cite[Theorem 10.2.1]{Joyce2} the only connected Lie subgroups of $\Gtwo$ that can arise as restricted holonomy of
the  Riemannian metric $g$ are $\Gtwo$, $\mathrm{SU}(3)$, $\mathrm{SU}(2)$ and $\{1 \}$.
Since $b_1 (\widetilde{M}) = 1$ and  $\pi_1 (\widetilde{M}) = \Z$,  the restricted holonomy  group of $g$ must be $\mathrm{SU}(3)$.

Therefore, $\widetilde{M}$ has a finite covering  $N\times S^1$ with $N$ being a $6$-dimensional simply connected 
Calabi--Yau manifold. Indeed, by Proposition 1.1.1 of \cite{Joyce1} we know that  $(\widetilde{M}, g)$ must admit 
as Riemannian finite cover a product  $N \times S^1$, for some compact, 
 simply connected 6-manifold $N$. Since the holonomy group
of the induced metric on the finite cover is the product of the holonomy group of $N$ and the trivial group, 
the induced metric on $N$ is Ricci-flat and its holonomy group is $\mathrm{SU}(3)$. That is $N$ is a Calabi--Yau manifold.

The deck transformations group of the covering map $N \times S^1 \to \widetilde{M}$
consists of maps which are products of an isometry of the Calabi--Yau
manifold $N$ and a rotation of finite order on $S^1$. This is due to the fact
that the deck transformations are isometries and that $N\x S^1$ is a
Riemannian product.
Furthermore, the above deck transformations are homotopic
to the identity. Therefore, $H^*(N \x S^1) \cong H^*(\widetilde{M})$
and  the minimal models  are the same. Thus, on $N \times S^1$ and, consequently on $\widetilde{M}$ there exist a closed $2$-form $\omega$
and a closed $1$-form $\eta$,
such that $[\omega]^3 \smallsmile [\eta]\not= 0$ in the cohomology of
$\widetilde{M}$. But this is not possible 
in the algebra $H^*(\widetilde{M})$. First, we see that it is not possible in $H^*(\widehat{M})$ since we must  
have $\eta=e^3$, and we know  that $H^2(\widehat{M})  = \la [e^{16}], [e^{25} + e^{34}]\ra.$ Then we use Proposition \ref{prop-cohom}.
\end{proof}

\begin{remark}
The proof of Theorem \ref{th:notorfree} also shows that   $\widetilde{M}$
cannot be a product of $S^1$ and a $6$-manifold.

Moreover, by (\ref{eqn:finishing}), the Poincar\'e polynomial of $\widetilde{M}$ is
 $$P(t)=(1-2t+21t^2-2t^3+t^4)(1+t)^3.$$
It can be checked that the factor of degree 4 is irreducible over the
rationals and the polynomial $P(t)$ does not factorize as the product of two
polynomials of degree greater than 1 with non-negative integer coefficients.
Therefore, $\widetilde{M}$ cannot be a product of two manifolds of
dimension greater than 1 too.
\end{remark}

%%%%%%%%%%%%%%%%%%%%%%%%%%%%%%%%%%%%%%%%%%%%%%%%%%%%%%%%%%%%%%%%%%%%%%%%%%
%%%%%%%%%%%%%%%%%%%%%%%%%%%%%%%%%%%%%%%%%%%%%%%%%%%%%%%%%%%%%%%%%%%%%%%%%%
%%%%%%%%%%%%%%%%%%%%%%%%%%%%%%%%%%%%%%%%%%%%%%%%%%%%%%%%%%%%%%%%%%%%%%%%%%
\section{Associative 3-folds in $\widetilde M$} \label{sect:3-folds}
%%%%%%%%%%%%%%%%%%%%%%%%%%%%%%%%%%%%%%%%%%%%%%%%%%%%%%%%%%%%%%%%%%%%%%%%%%
%%%%%%%%%%%%%%%%%%%%%%%%%%%%%%%%%%%%%%%%%%%%%%%%%%%%%%%%%%%%%%%%%%%%%%%%%%

The closed $\Gtwo$ form $\widetilde\varphi$ constructed on~$\widetilde M$
defines an {\em associative calibration} on~$\widetilde M$. This means that, for any $p\in\widetilde  M$, we have that
every oriented 3-dimensional subspace $V$ of the tangent space
$T_p \widetilde M$ satisfies
$\widetilde\varphi(p)|_V =\lambda\operatorname{vol}_V$, for some
$\lambda\leq 1$, where the volume form $\operatorname{vol}_V$ is
induced from the restriction to $V$ of the inner product
$g_{\widetilde\varphi}$ at~$p$ (see \cite{HarveyLawson} and \cite[\S 3.7]{Joyce2}).
The 3-dimensional orientable submanifolds $Y\subset\widetilde M$ {\em calibrated
by the} $\Gtwo$ {\em form} $\widetilde\varphi$, i.e.\ those submanifolds $Y\subset\widetilde M$ that satisfy 
$\widetilde\varphi(p)|_{T_p Y}=\operatorname{vol}_Y(p)$, 
for each $p\in Y$ and for some unique orientation of $Y$, are often called
{\em associative 3-folds}. Every compact calibrated submanifold
$Y$ is volume-minimizing in its homology class, in particular $Y$ is
minimal~\cite[Proposition 3.7.2]{Joyce2}.

We shall produce examples of associative 3-folds in $\widetilde M$ from the
fixed locus of a $\Gtwo$-involution of the compact manifold $M=\Gamma{\backslash} G$ defined in~\eqref{def:nilmfd}, 
applying the following.

\begin{proposition}[{\cite[Proposition 10.8.1]{Joyce2}}]\label{g2.involution} 
Let $N$ be a 7-manifold with a closed $\Gtwo$ form $\phi$, and let 
$\sigma:N\to N$ be an involution of $N$ satisfying $\sigma^* \phi =\phi$ and such that
$\sigma$ is not the identity map.
Then the fixed point set $P=\{p\in N\, \vert \,\sigma(p)=p\}$ is an embedded associative 3-fold. Furthermore, if $N$ is
compact then so is~$P$.
\end{proposition}

\begin{remark}\label{rem:invol-orbifolds}
Note that Proposition 10.8.1 in \cite{Joyce2} is stated for the $\Gtwo$-structures that are closed
and coclosed, but the coclosed condition is not used in the proof. 
\end{remark}

Recall from section~\ref{sec:closedG2-orbifold} the 7-dimensional Lie group
$G$, and consider on~$G$ the involution given by
\begin{equation}\label{eqn.sigma}
\sigma:
(x_{1}, x_{2}, x_{3}, x_{4}, x_{5}, x_{6}, x_{7}) \mapsto
(- x_{1}, - x_{2}, x_{3}, x_{4}, - x_{5}, \textstyle{\frac12} - x_{6}, x_{7}).
\end{equation}
The involution $\sigma$ is equivariant with respect to the left multiplications by 
elements of the subgroup $\Gamma\subset G$. Indeed, for each $a\in G$ and
$A\in\Gamma$ we may write, noting the properties of the $\ZZ_2$-action $\rho$ on $G$ defined by \eqref{eqn:rho},
\begin{multline*}
L_A(\sigma(a))=L_A(L_{\frac12}(\rho(a)))=L_{\frac12}(L_A(\rho(a)))\\
= L_{\frac12}(\rho(A')\cdot\rho(a))=L_{\frac12}(\rho(L_{A'}(a)))=\sigma(L_{A'}(a)),
\end{multline*}
where $A'=\rho(A)$ and $L_{\frac12}$ denotes the left translation by an
element with coordinates $(x_i)=(0,0,0,0,0,\frac12,0)$ in~$G$.
Therefore, $\sigma$ descends to the quotient manifold $M=\Gamma{\backslash} G$.
The induced map on~$M$, still denoted by $\sigma$, commutes with $\rho$
and so $\sigma$ descends to the orbifold $\widehat{M}=M/{\mathbb{Z}}_2$.
{}From now on, we denote by $\widehat{\sigma}$ the involution of $\widehat{M}$ induced by $\sigma$.

The fixed locus $\widehat{P}$ of $\widehat{\sigma}$ is the 
image by the natural projection $\widehat{\pi} \colon M \to \widehat M$ of the 
set $P$ of points in $M$ that are fixed by the involution $\sigma:M\to M$
or by $\sigma\circ\rho=L_{\frac12}:M\to M$. Thus,
$\widehat{P}$ consists of all the $3$-dimensional spaces
 $\widehat{P}_{\mathbf{b}} \,=\,\widehat{\pi}(P_{\mathbf{b}})=P_{\mathbf{b}}/{\mathbb{Z}}_2$, where
$$
 P_{\mathbf{b}}\,=
\begin{cases}
\{\Gamma\cdot (b_1, b_2, x_3, x_4, b_5, b_6, x_7)\, \vert \, x_3, x_4, x_7
\in \R \} \subset M, & \text{ if }b_1=0\\
\{\Gamma\cdot (b_1, b_2, x_3, x_4, b_5, \frac32 b_2 + b_6 - x_4,  x_7)\, \vert \, x_3, x_4, x_7
\in \R \} \subset M, & \text{ if }b_1=1\\
\end{cases}
$$  
 and
  $$
  \mathbf{b}=(b_1, b_2, b_5, b_6) \in \mathbb{B}=\{0,1\} \x \{0, 1/2\}\x \{0, 1/2\}\x \{1/4, 3/4\}.
 $$ 
Hence, $P$ is a disjoint union of 16 copies of a $3$-torus $T^3$. 
Now one can check that the fixed locus $\widehat{P}$ of $\widehat{\sigma}$ consists of 8 disjoint copies of~$T^3$ since 
in the orbifold $\widehat{M}$ the points of coordinates $(b_1, b_2, x_3, x_4, b_5, 1/4, x_7)$ and $(b_1, b_2, x_3, x_4, b_5, 3/4, x_7)$ 
are the same.
Observe that the fixed loci $P$ of $\sigma$ and $S'$ of $\rho$ do not intersect, and hence the fixed locus $\widehat{P}$ of 
$\widehat{\sigma}$ and the singular locus $S$ of the orbifold $\widehat{M}$ also do not intersect.

\begin{proposition}\label{prop:3-folds}
Each of the eight disjoint copies of $3$-tori in $\widehat M$, which are the fixed locus $\widehat{P}$ of $\widehat{\sigma}$,
define eight embedded, associative (calibrated by~$\widetilde\varphi$), minimal $3$-tori in~$\widetilde M$.
\end{proposition}
\begin{proof}
Since the $\Gtwo$ form $\varphi$ on $M$ defined in~\eqref{eqn:closed-nilmfd} is
preserved by the involution $\sigma$ of $M$, each of the 16 tori $P_{\mathbf{b}}$ in $M$  fixed by
$\sigma$ is an associative $3$-fold in $(M, \varphi)$ by Proposition~\ref{g2.involution}.
Now we know that the $\ZZ_2$-action $\rho$ on~$M$ preserves the $\Gtwo$ form $\varphi$ on $M$, 
and induces the $\Gtwo$ form $\widehat{\varphi}$ 
on $\widehat{M}$ (see section~\ref{sec:closedG2-orbifold}), so that the pull-back of 
$\widehat{\pi}$ sends $\widehat{\varphi}$ to $\varphi$. Thus, the 2-to-1 projection map 
$\widehat{\pi}\colon M \to \widehat M$ outside the set $S'$ of points in
$M$ fixed by $\rho$ is a local isomorphism of the
closed $\Gtwo$-structures and hence also a local isometry of the induced
metrics.  Consequently, $\widehat{\pi}$ preserves the associative calibrated property of submanifolds, and so each of the eight
copies of~$T^3$ is an associative (and minimal) 3-fold in~$\widehat M$. 
Furthermore, as we mentioned above, these $3$-tori do
not meet the singular locus $S$ of $\widehat{M}$. 

To complete the proof, let us recall that the
$\Gtwo$-structure $\widetilde{\varphi}$ on $\widetilde M$ agrees, away
from a neighbourhood $U$ of $S$, with the $\Gtwo$-structure 
$\widehat{\varphi}$ induced on $\widehat{M}$ from $M$. It follows that the
above $3$-tori lift diffeomorphically to the resolution~$\widetilde M$ and
define 8 embedded, associative (calibrated by~$\widetilde\varphi$), minimal
$3$-tori in~$\widetilde M$. 
\end{proof}

McLean~\cite{McLean} studied the deformation problem for several types of
calibrated submanifolds. For compact associative 3-folds, the problem
may be expressed as a non-linear elliptic PDE, with index zero, if the $\Gtwo$
form is closed and coclosed. This result was generalized by Akbulut and Salur
to arbitrary, not necessarily closed or coclosed, $\Gtwo$ forms \cite[Theorem 6]{Akbulut-Salur}. It follows that any
compact associative 3-fold in $\widetilde M$ is either rigid or, otherwise, has
infinitesimal associative deformations which in general need not arise from
the actual deformations (as the linear part of the deformation problem may
have a nontrivial cokernel).

As we now show, the $3$-tori in the present example do have associative
deformations.
\begin{proposition} \label{prop:deformations 3-folds}
Each of the eight associative $3$-tori in $\widetilde M$ arising from the fixed
locus of~$\sigma$ has a smooth $3$-dimensional family of
non-trivial associative deformations.
\end{proposition}
\begin{proof}
As in the previous sections, in light of the symmetry by left translations,
it suffices to consider just one component $Y_0$ of the fixed locus
of~$\sigma$. A tubular neighbourhood of $Y_0$ in $\widetilde M$ is isometric
to a tubular neighbourhood of the image of $Y_0$ in the smooth locus
$\widehat{M}\setminus S$. As the projection $M\to \widehat{M}$ is a local
isometry away from the preimage of~$S$ we may work on $M$ with the
$\Gtwo$-structure $\varphi$ and consider a component of the preimage of
$Y_0$ which by abuse of notation we continue to denote by~$Y_0\subset M$.
We may choose $Y_0$ to be defined by $x_1=x_2=x_5=0$, $x_6=\frac14$, then
the associative 3-torus $Y_0$ is contained in the fiber $p^{-1}(0+2\ZZ)$
of the projection $p:M\to \RR/2\ZZ$ (see~\eqref{eqn: projection p}).

Every fiber $p^{-1}(x_1)$ has a natural structure of a complex $3$-torus
$\CC^3/\Lambda(x_1)$, where the complex coordinates on $\CC^3$ are given by
$x_2+ix_3, \  x_4+ix_5, \ x_6+ix_7$. Moreover, these latter 3-tori
are biholomorphic to the standard $3$-torus $p^{-1}(0)=\CC^3/\ZZ^6$ because
the linear isomorphisms $B(x_1)$ and $C$ from Lemma~\ref{lem:mapping-torus}
are contained in the image of 
$\SL(3,\RR)$ under the chain of natural embeddings of groups
$\SL(3,\RR)\subset \SL(3,\CC)\subset \SL(6,\RR)$. 
The complex $3$-form
$\Omega=(e^2+ie^3)\wedge (e^4+ie^5)\wedge (e^6+ie^7)$ induces, via the
pull-back,
on each complex torus $p^{-1}(x_1)$ a holomorphic trivialization of
the canonical bundle of $(3,0)$-forms. The (pull-back of the)
closed 2-form $\omega=e^2\wedge e^3 + e^4\wedge e^5 + e^6\wedge e^7$ induces
on $p^{-1}(x_1)$ a Ricci-flat K\"ahler metric which depends non-trivially on
$x_1$ and when $x_1=0$ coincides with the `usual' K\"ahler metric on
$\CC^3/\ZZ^6$. Thus each fiber $p^{-1}(x_1)$ has a torsion-free $\SU(3)$
(Calabi--Yau) structure compatible with the closed
$\Gtwo$-structure
$\varphi=e^1\wedge\omega-\Re\Omega$  on~$M$, in the sense that
$\iota_{x_1}^*\omega = \iota_{x_1}^*(dx_1 \lrcorner \varphi)$ and
$\iota_{x_1}^*\Re\Omega = \iota_{x_1}^*\varphi$,
where $\iota_{x_1}:p^{-1}(x_1)\to M$ denotes the embedding.

It is not difficult to check that $Y_0$ is a special Lagrangian $3$-torus in the
Calabi--Yau threefold $Z_0=p^{-1}(0)$. Furthermore, the special Lagrangian tori
\begin{equation}\label{eqn:assoc.deforms}
Y(a,b,c)=\{(a,y_1,y_2,b,{\textstyle\frac14}+c,y_3)+\ZZ^6\, | \,
(y_1,y_2,y_3)\in\RR^3\}
\end{equation}
in $p^{-1}(0)$ are associative in $(M,\varphi)$ as
$\varphi|_{Y(a,b,c)}=dx_3\wedge dx_4\wedge dx_7|_{Y(a,b,c)}=
e^3\wedge e^4\wedge e^7|_{Y(a,b,c)}$.
For small $a,b,c$, the $Y(a,b,c)$ induce well-defined non-trivial associative
deformations of $Y_0$ in~$(\widetilde M,\widetilde\varphi)$. 
\end{proof}

We next show that the result of Proposition~\ref{prop:deformations 3-folds} is
optimal. For this, we require some
foundational results about the deformations of associative 3-folds.

Let $N$ be a 7-manifold with a $\Gtwo$-structure $\phi$.
Denote by $\chi_\phi$ the $3$-form on $N$ with values in $TN$ determined by
$$
\la\chi_\phi(u,v,w),a\ra=\star_\phi \phi(u,v,w,a),
$$
for all $u,v,w,a\in TN$. The $3$-form $\chi_\phi$ may also be locally expressed
as
$$
\chi_\phi=\sum_{k=1}^7 (e_j\hook {\star}_\phi\phi)\otimes e_j,
$$
for any local positive-oriented orthonormal frame field $\{e_j\}$ on $N$
\cite[p.~1217]{Gayet}.

For $P$ an oriented $3$-dimensional submanifold of $N$, let $\omega_P$ denote a
global section of $\Lambda^3 TP$ given by $f_1\wedge f_2\wedge f_3$, for any
local positive orthonormal frame field $\{f_k\}$ on~$P$. It can be checked
that then $\chi_\phi(\omega)$ is a section of the 
normal vector bundle ${\mathcal{N}}_{P/N}$ of $P$ in $N$.
Furthermore, the submanifold $P$ will be associative with respect to $\phi$
(and calibrated when $\phi$ is closed) if and only if $\chi_\phi(\omega_P)=0$
(cf. \cite{HarveyLawson} or~\cite{McLean}).

Now, let us consider $P$ a compact associative $3$-fold with respect to $\phi$.
It is by now a standard consequence of
the tubular neighbourhood theorem that smooth local deformations of $P$ may be
given by $P(\mathbf{v})=\exp_\mathbf{v} (P)$ for smooth sections $\mathbf{v}$
of the normal vector bundle 
${\mathcal{N}}_{P/N}$ of $P$ in $N$ with $\|\mathbf{v}\|_{C^0}$ small, where
the exponential map and the $C^0$ norm are defined using the metric~$g_\phi$.
For every \mbox{$C^0$-small} normal vector field
$\mathbf{v}\in\Gamma(\mathcal{N}_{P/N})$ and every closed $\Gtwo$-structure
$\phi$ on $N$, define the `deformation map'
\begin{equation}\label{deformation.map}
F(\mathbf{v},\phi)=(\exp_{\mathbf{v}}^* \chi_\phi)(\omega_P)
\in\Gamma ({\mathcal{N}}_{P/N,\; \phi}),
\end{equation}
where the normal bundle ${\mathcal{N}}_{P/N,\; \phi}$
is defined using the metric $g_{\phi}$.
Then $P(\mathbf{v})$ will be associative calibrated by $\phi$ precisely when
$F(\mathbf{v},\phi)=0$.

\begin{proposition}\label{prop:maximal}
In the case when $P=Y$ is one of the eight associative $3$-tori given in
Proposition~\ref{prop:deformations 3-folds}, the kernel of the derivative
$D_1 F|_{(0,\widetilde\varphi)}$ of the map~\eqref{deformation.map} in the
first argument has dimension $3$.
\end{proposition}
\begin{corollary}\label{cor:maximal}
The family~\eqref{eqn:assoc.deforms} of associative local deformations of $Y$
is {\em maximal}  (that is, it is not contained as a proper subset in another associative
local deformation family).
\end{corollary}
\begin{proof}[Proof of Proposition~\ref{prop:maximal}]
For the same reason as in the proof of Proposition~\ref{prop:deformations
3-folds}, we may take $P$ to be the associative $3$-torus
$Y_0$ in $M$ with the closed $\Gtwo$-structure $\varphi$. Thus $Y_0$ is
defined by $x_1=x_2=x_5=0$, $x_6=\frac14$ and $x_3,x_4,x_7\in \R/\Z$ define
the local coordinates on $Y_0$.

It is easy to check that the frame field $e_i$ dual to $e^i$ on $M$ (see
\eqref{eqn:struct-eq}) is given, in the local coordinates $x_i$ induced
from $G$, by
\begin{equation}\label{eqn:orthon.frame-field}
\begin{split}
&e_1=\bd_1+x_2\bd_4+x_3\bd_5-x_1 x_2\bd_6-x_1 x_3\bd_7,\quad
e_2=\bd_2,\quad
e_3=\bd_3,
\\
&e_4=\bd_4-x_1\bd_6,\qquad
e_5=\bd_5-x_1\bd_7,\qquad
e_6=\bd_6,\qquad
e_7=\bd_7,
\end{split}
\end{equation}
where $\bd_i=\frac\bd{\bd x_i}$ denote the local coordinate vector fields.

The restrictions to $Y_0$ of the vector fields $e_i$, $i=1,2,5,6$,
give an orthonormal frame field inducing a trivialization of the normal
bundle ${\mathcal{N}}_{Y_0/ M,\; \varphi}$ and the restrictions of
$e_k$, $k=3,4,7$ define an orthonormal frame field on $Y_0$ trivializing
the tangent bundle $TY_0$.

The linear operator in question acting on the sections of
${\mathcal{N}}_{Y_0/ M,\; \varphi}$ and may be expressed (see
\cite{Akbulut-Salur},\cite[Theorem~2.1]{Gayet}) as
\begin{equation}\label{eqn:twisted-Dirac}
D_1(\mathbf{v})=D_1 F|_{(0,\widetilde\varphi)}(\mathbf{v})=
\sum_{k=3,4,7} e_k\times\nabla_{e_k}^\bot \mathbf{v} +
\sum_{i=1,2,5,6}(\nabla_{\mathbf{v}} {\star}_{\varphi}\varphi)
(e_i,\omega_Y)\otimes e_i,
\end{equation}
where $\times$ denoted the octonionic cross-product corresponding to the
$\Gtwo$-structure $\varphi$,
$\nabla^\bot$ is the connection on ${\mathcal{N}}_{Y_0/ M}$  induced by the
Levi-Civita connection $\nabla$ of~$g_{\varphi}$ and
$\omega_Y=e_3\wedge e_4\wedge e_7\in\Gamma(\Lambda^3TY_0)$.
The second sum in~\eqref{eqn:twisted-Dirac} contains the terms arising
from the failure of the $\Gtwo$-structure to be torsion-free and does not
contain derivatives of $\mathbf{v}$. On the other hand, the first sum
in~\eqref{eqn:twisted-Dirac} is a Dirac-type operator arising in McLean's
results~\cite[\S 5]{McLean}.

In the present case, we may consider $D_1$ as a first order differential
operator acting on functions $\mathbf{v}=(v_1,v_2,v_5,v_6)$, where each
$v_i(x_3,x_4,x_7)$ is periodic with period 1 in each variable~$x_k$. It is
not difficult to check that the first order terms in $D_1$ are equivalent
to the standard `flat space' Dirac operator given in terms of the Pauli
spin matrices. The zero order terms may be determined by a straightforward,
albeit lengthy computation; the following table gives some check-points for
the readers convenience.
\begin{table}[h]
\caption{The values of $2\,\nabla^\bot_{e_k} e_i$ on $Y_0$,
for $i=1,2,5,6$, $k=3,4,7$.}
\begin{center}\begin{tabular}{c|cccc}
&1&2&5&6\\
\hline
3&$e_5$    &0     &$-e_1$&0\\
4&$e_2+e_6$&$e_1$&0     &$-e_1$\\
7&$e_5$    &0     &$-e_1$&0
\end{tabular}\end{center}
\end{table}
\begin{table}[h]
\caption{The values of $2\,\nabla_{e_i} e^j$ on $Y_0$,
for $i=1,2,5,6$, $j=1,\ldots,7$.}
\begin{center}\begin{tabular}{c|ccccccc}
&1&2&3&4&5&6&7\\
\hline
1&0&$-e^4$&$-e^5$&$-e^6$&$e^3-e^7$&$e^4$&$e^5$\\
2&$-e^4$&0&0&$-2e^1$&0&0&0\\
5&$e^3+e^7$&0&$-e^1$&0&0&0&$-e^1$\\
6&$e^4$&0&0&$-e^1$&0&0&0\\
\end{tabular}\end{center}
\end{table}

We then obtain
$$
D_1: \mathbf{v}=
\begin{pmatrix} v_1\\ v_2\\ v_5\\ v_6 \end{pmatrix}
\mapsto
\begin{pmatrix}
0        &-\bd_3 &\bd_4  &-\bd_7\\
\bd_3-1  &0      &-\bd_7 &-\bd_4\\
-\bd_4   &\bd_7  &0      &-\bd_3\\
\bd_7+1  &\bd_4  &\bd_3  &0
\end{pmatrix}
\begin{pmatrix} v_1\\ v_2\\ v_5\\ v_6 \end{pmatrix}
$$

By considering the Fourier expansions of $v_i$, we find that the kernel of
$D_1$ consists of constant vectors and is spanned by $e_i|_{Y_0}$, for
$i=2,5,6$. The latter corresponds to the tangent space at $Y_0$ to the
$3$-dimensional family of associative deformations given
in~\eqref{eqn:assoc.deforms}.
\end{proof}

We also show, as a direct consequence of the next result, that the associative
$3$-tori $Y$ in Proposition~\ref{prop:3-folds} become rigid after a suitable
arbitrary small perturbation of the closed $\Gtwo$-structure
$\widetilde\varphi$ on an arbitrary small neighbourhood of $Y$ in $\widetilde M$. 

If $P$ is a compact associative $3$-fold with respect to a closed $\Gtwo$-structure $\phi$ on a $7$-manifold $N$,
denote by $\mathcal{M}_{P,\phi}$ the set of smooth
associative $3$-folds calibrated by $\phi$ and isotopic to~$P$.
The next proposition is a rather general result and possibly of independent
interest; it is not specific to the particular construction in this paper.

\begin{proposition}\label{prop:isolate.associative}
Let $\phi$ be a closed $\Gtwo$ form on a $7$-manifold~$N$ and $P\subset N$
a compact associative $3$-fold calibrated by $\phi$. Suppose that the kernel
of $D_1 F|_{(0,\phi)}$ is spanned by the normal vector fields
$\mathbf{f}_1,\ldots,\mathbf{f}_m$, where $1\le m\le 4$ and $\mathbf{f}_j$ are
linearly independent at each point of~$P$. Then there is a
neighbourhood $U$ of $P$ and a closed $\Gtwo$ structure $\psi$ with arbitrary
small $\|\psi-\phi\|_{C^0}$ (defined using the metric $g_\phi$), such that the
only element of $\mathcal{M}_{P,\psi}$ contained in $U$ is $P$.
\end{proposition}
\begin{proof}[Sketch-proof]
We claim that the argument of Gayet in \cite[Proposition~2.6]{Gayet} adapts to the
present situation to give a proof of Proposition~\ref{prop:isolate.associative}.
The only difference from the hypotheses of Gayet's result is that in the present
case the kernel of $D_1$ is spanned by up to four, rather than one,
non-vanishing vector fields.

We give a brief review of the proof in \cite{Gayet} with a modification
for a $m$-dimensional kernel of $D_1$. The normal bundle of an associative
$3$-fold is always trivial (e.g.~\cite[Remark~2.14]{JoyceKarigiannis}) and a
tubular neighbourhood of $P$ may be chosen diffeomorphic to $P\times\R^4$ with
$u_i$, $i=1,2,3,4$, the coordinates on $\R^4$ such that
$\bd/\bd u_j=\mathbf{f}_j$ for $j=1,\ldots,m$. For each $j$, we may write
$\mathbf{f}_j\hook {\star}_\phi\phi=\sum_{i\neq j} du_i\wedge \beta_{ji}$, for some
2-forms $\beta_{ji}$ and define $\psi_j=d(u_j\sum_{i\neq j} u_i \beta_{ji})$.
For every
$\lambda=(\lambda_1,\ldots,\lambda_m)\in\R^m$ close to zero, the 3-form
$\phi_\lambda=\phi+\sum_{j=1,\ldots,m}\lambda_j\psi_j$ gives a well-defined
$\Gtwo$-structure such that $\phi_\lambda|_P=\phi|_P$. Thus $P$ is associative
with respect to $\phi_\lambda$ for each $\lambda$. Let $D_1^\lambda$ be the
derivative in the first variable of the deformation
map~\eqref{deformation.map} associated 
with $\phi_\lambda$. It suffices to prove that $D_1^\lambda$ is injective and
then Proposition~\ref{prop:isolate.associative} will follow from McLean's
theory by application of the implicit functions theorem in Banach spaces as
$D_1^\lambda$ is a self-adjoint elliptic operator of index zero
(cf.~\cite[\S 5]{McLean},\cite[Proposition.~2.2]{Gayet}).

Let $\mathbf{v}=\mathbf{v}_0+\sum_{i=m+1}^4\frac{\bd}{\bd u_i}$
and $\mathbf{v}_0=\sum_{j=1}^m v_j\mathbf{f}_j$. It follows from the proof of
\eqref{eqn:twisted-Dirac} in \cite{Gayet} that the operator $D_1^\lambda$
admits an expansion of the form
$$
D_1^\lambda \mathbf{v} = D_1 \mathbf{v} + \sum_{j=1}^m\lambda_j v_j \mathbf{f}_j
+ O(|\lambda| (v_i)_{i=m+1,\ldots,4}) + O(|\lambda|^2 \mathbf{v})
$$
for small $\lambda$ and $\mathbf{v}$. The key point in the argument of
\cite[Proposition.~2.6]{Gayet} is an application of the elliptic theory to show
that if $D_1^\lambda\mathbf{v}=0$, then
$\nabla \mathbf{v}_0=O(|\lambda|\mathbf{v})$ and
the norms of $\mathbf{v}$ and $\mathbf{v}_0$ are Lipschitz equivalent
and
\begin{equation}\label{D1-expansion}
D_1\mathbf{v}=-\sum_{j=1}^m\lambda_j v_j \mathbf{f}_j + O(|\lambda|^2 \mathbf{v}).
\end{equation}

On the other hand, by considering an elliptic estimate for $(D_1)^2$
one can deduce that $D_1\mathbf{v}=O(|\lambda|^2 \mathbf{v})$ thus obtaining
a contradiction with \eqref{D1-expansion} for any small non-zero $\lambda_j$,
unless $\mathbf{v}=0$.
\end{proof}

%%%%%%%%%%%%%%%%%%%%%%%%%%%%%%%%%%%%%%%%%%%%%%%%%%%%%%%%%%%%%%%%%%%%%%%%%%
%%%%%%%%%%%%%%%%%%%%%%%%%%%%%%%%%%%%%%%%%%%%%%%%%%%%%%%%%%%%%%%%%%%%%%%%%%
%%%%%%%%%%%%%%%%%%%%%%%%%%%%%%%%%%%%%%%%%%%%%%%%%%%%%%%%%%%%%%%%%%%%%%%%%%
\section{A coassociative torus fibration}  \label{sect:4-folds}
%%%%%%%%%%%%%%%%%%%%%%%%%%%%%%%%%%%%%%%%%%%%%%%%%%%%%%%%%%%%%%%%%%%%%%%%%%
%%%%%%%%%%%%%%%%%%%%%%%%%%%%%%%%%%%%%%%%%%%%%%%%%%%%%%%%%%%%%%%%%%%%%%%%%%

In this section, we shall consider a special class of 4-dimensional
submanifolds, which are defined on each $7$-manifold $N$ with a $\Gtwo$-structure defined 
by a $3$-form $\phi$. We may write $\phi= e^{123} +  e^{145} +
e^{167} - e^{246} + e^{257} + e^{347} + e^{356}$
(cf.~\eqref{eqn:closed-nilmfd}) and then the local co-frame field
$\{e_1,\dotsc, e_7\}$ is orthonormal in the metric $g_\phi$ (induced by~$\phi$)
and also positively oriented. The Hodge dual of $\phi$ is therefore
$$
\theta={\star}\phi=e^{4567}+e^{2367}+e^{2345}-e^{1357}+e^{1346}+e^{1256}+e^{1247}.
$$
The $4$-form $\theta$ satisfies $\theta(p)|_W =\lambda\operatorname{vol}_W$
with some $\lambda\leq 1$, for each $p\in N$ and every oriented
$4$-dimensional subspace $W$ of the tangent space $T_p N$. Here the volume
form $\operatorname{vol}_W$ is induced from the restriction to $W$ of the
inner product $g_{\phi}$ at~$p$ (cf.\ section~\ref{sect:3-folds}).
The orientable $4$-dimensional submanifolds $X\subset N$ satisfying 
\mbox{$\theta(p)|_{T_p X} =\operatorname{vol}_X(p)$,}
for all $p\in X$ and for some unique orientation of $X$, are called
{\em coassociative 4-folds}. The latter condition on~$X$ is equivalent to
$\phi|_X=0$. Note that if the $4$-form $\theta$ is not closed, then $\theta$
is not a calibration and the coassociative submanifolds of $N$
need not be minimal.

Once again, we start with the $7$-manifold $M$ with the $\Gtwo$-structure
$\varphi$ defined in section~\ref{sec:closedG2-orbifold}
(see~\eqref{eqn:closed-nilmfd} and Lemma~\ref{lem:mapping-torus})
and also use the orthonormal frame field $e_i$ dual to $e^i$ on $M$,
\begin{equation}
\begin{split}
&e_1=\bd_1+x_2\bd_4+x_3\bd_5-x_1 x_2\bd_6-x_1 x_3\bd_7,\quad
e_2=\bd_2,\quad
e_3=\bd_3,
\\
&e_4=\bd_4-x_1\bd_6,\qquad
e_5=\bd_5-x_1\bd_7,\qquad
e_6=\bd_6,\qquad
e_7=\bd_7.
\end{split}
\end{equation}
Observe that, in particular, the vectors $e_4,e_5,e_6,e_7$ span the same
$4$-dimensional subspace as $\bd_4,\bd_5,\bd_6,\bd_7$ and this subspace is
coassociative at each point of~$M$.
Furthermore, this latter subspace is invariant under the action of the 
linear isomorphisms $B(x_1)$, $C$ and $E$ in Lemma~\ref{lem:mapping-torus}.

Recall from section~\ref{sect:3-folds} that the fibers $p^{-1}(x_1)$ of the map
$p:M\to S^1$ have the structure of Calabi--Yau complex 3-tori $\CC^3/\Lambda(x_1)
\cong\CC^3/\ZZ^6$ compatible with the $\Gtwo$-structure on~$M$. We find that
for each $x_1\in\RR/2\ZZ$  and $w\in\CC/\ZZ^2$ the complex $2$-torus
$$
X_{x_1,w}=\{(w,z_2,z_3)\in p^{-1}(x_1):(z_2,z_3)\in\CC^2/\ZZ^4\}
$$
is a well-defined complex submanifold of $p^{-1}(x_1)$ and a
coassociative $4$-fold in~$M$. (Here we used $w=x_2+ix_3$,
$z_2=x_4+ix_5$, $z_3=x_6+ix_7$ to denote the complex coordinates.)
The tori $X_{x_1,w}$ are the fibers of a coassociative fibration map
\begin{equation}
\begin{split}
&q:M\to T^3=(\RR/2\ZZ)\times(\RR^2/\ZZ^2).
\\
&[(x_1,\ldots, x_7)]  \mapsto  (x_1 + 2\ZZ,x_2+\ZZ,x_3+\ZZ)
\end{split}
\end{equation}
Note that in the definition of $q$ we use the local coordinates $\{x_i\}$ 
on the nilmanifold $M$ defined by~\eqref{def:nilmfd}.

The fibers of $q$ may be considered as a deformation family. 
By McLean's theorem the local deformations of a compact coassociative
$4$-fold $X$ in a $7$-manifold with a closed $\Gtwo$-structure form a smooth
manifold of dimension $b^2_+(X)$ \cite{McLean}. (McLean stated this result for
torsion-free $\Gtwo$-structures but his argument only uses the closed condition,
as was subsequently observed by Goldstein~\cite{goldstein}.) A $4$-torus has
$b^2_+=3$, therefore the fibers of~$q$ form a {\em maximal} deformation family
of coassociative $4$-folds.

The map $q$ is $\rho$-equivariant (with a natural involution induced by
$\rho$ on the image of~$q$) and induces a coassociative fibration of
the orbifold
$$
\widehat{q}:\widehat{M}\to (T^2/\pm1)\times S^1 \simeq S^2\times S^1
$$
with the $S^2$ factor understood as an orbifold (sometimes referred to as the
`pillowcase') homeomorphic to the standard 2-sphere. When $x_1\in\{0,1\}$
and $x_2\in\{0,\frac12\}$ the fiber of $\widehat{q}$ is a singular orbifold
homeomorphic to $S^2\times T^2$ (with $S^2$ again understood as the
`pillowcase').

We next show that the map $\widehat{q}$ lifts to $\widetilde M$ and induces
a coassociative fibration $\widetilde q$ on $\widetilde M$, so that there
is a commutative diagram
\begin{equation}\label{diagram}
\begin{CD}
\widetilde{M} @>{\pi}>> \widehat{M}\\
@V{\widetilde q}VV                @VV{\widehat q}V\\
S^2\times S^1 @>{\simeq}>> (T^2/\pm 1)\times S^1
\end{CD}
\end{equation}
where the horizontal arrows are, respectively, the resolution (blow-up)
$\widetilde M\to\widehat M$ and the `pillowcase homeomorphism'.

In order to construct the desired $\widetilde q$, we first deduce from the
construction of $\widehat q$ that every fiber passing through a singular
locus of $\widehat M$ has singular points. A neighbourhood of each
singular point of this singular fiber is diffeomorphic (in the orbifold
sense) to a neighbourhood of $(T^2/\pm1)\times T^2$ in
$(\RR\times T^3/\pm 1)\times T^3$, for
suitable embeddings $T^2\to \RR\times T^3$ and $T^2\to T^3$. For example,
near the (equivalence class of) the zero vector in $\RR^7$ the embeddings are
induced by $(x_5,x_6)\mapsto (x_1,x_2,x_5,x_6)$ and
$(x_4,x_7)\mapsto(x_3,x_4,x_7)$, where as usual the local coordinates $x_i$
on $\widehat M$ correspond to the local coordinates on the compact manifold~$M$.

As $\pi\colon \widetilde{M} \to \widehat{M}$ is a diffeomorphism away from the
preimage of the singular locus
of $\widetilde M$, it is easy to see that a generic fiber of $\widetilde q$
will be diffeomorphic to the 4-torus. But there will also be singular fibers.

We can understand the singular fibers of $\widetilde q$ via the local model
of the complex surface $S\to\CC^2/\pm 1$ defined by blowing up the singular
point of the cone $\CC^2/\pm 1$. Here the complex coordinates on $\CC^2$
correspond to $\zeta_1=x_4+ix_7$ and $\zeta_2=x_5+ix_6$ in the above notation.
Consider the cone $\CC/\pm 1\subset \CC^2/\pm 1$, where $\CC$ is understood
as a complex line with coordinate $\zeta_2$ in $\CC^2$ passing
through the origin. The proper transform of this cone is a non-singular
complex curve passing through the exceptional divisor on~$S$.
The inverse image of $\CC/\pm 1$ in~$S$ is the union of the latter complex
curve and the exceptional divisor (which is a copy of $\CP^1$) over the
singular point. We find that the lifted fiber in~$\widetilde M$ has a
singularity, locally modeled on the intersection of two copies of $\RR^4$
along $\RR^2$.

We claim
\begin{lemma}
The fibers $X$ of $\widetilde q$ are coassociative in the $\Gtwo$-structure
$\widetilde\varphi$ on~$\widetilde M$.
\end{lemma}
\begin{proof}
Firstly, observe that the defining condition $\widetilde\varphi|_X=0$ for a
$4$-dimensional submanifold $X$ to be coassociative is point-wise and linear in
the $3$-form $\widetilde\varphi$. Now recall that at each point in the
resolution region in $\widetilde M$ the $\Gtwo$-structure $\widetilde\varphi$ is
a linear combination of $\widehat\varphi$ induced from $\widehat M$ and the
$\Gtwo$-structure corresponding to the Riemannian product of the 3-torus and a
Ricci-flat K\"ahler complex surface. The $3$-form of former $\Gtwo$-structure
vanishes on the fibers of $\widetilde q$ by the above discussion, and the
$3$-form on the latter $\Gtwo$-structure vanishes on the fibers because each
relevant fiber is a Riemannian product with a special Lagrangian factor (or a
complex factor, depending on which complex structure is considered) in the
latter complex surface.
\end{proof}

The structure of a neighbourhood of a singular fiber in $\widetilde M$
is in fact a suspension over the familiar elliptic
fibration of a Kummer $K3$ surface. When one blows up the singular points
of $T^4/\pm 1=(\CC^2/(\ZZ+i\ZZ)^2)/\pm 1$, the proper transform of
$T^2/\pm 1=(\CC/(\ZZ+i\ZZ))/\pm 1$ is a non-singular complex curve, whereas
the inverse image of $T^2/\pm 1$ is a singular complex curve which is an image
of a (non-bijective) immersion of the Riemann sphere~$S^2$. This immersion
takes two distinct points to the same point in the image and is one-to-one
elsewhere on $S^2$. We find that, respectively, each singular fiber
in~$\widetilde M$ is the image of a {\em non-singular} 4-manifold
$T^2\times S^2$ under an immersion, intersecting itself along~$T^2$. The
singular fibers occur in one-dimensional families parameterized by~$S^1$, with
coordinate $x_1$ (as in the model example above).

Results of the deformation theory in~\cite{McLean} remain valid for
coassociative immersions of compact smooth $4$-manifolds. Notice that, as
$b^2_+(T^2\times S^2)=1$, the latter $S^1$-families of singular fibers are
{\em maximal} deformation families.

To summarise, we obtain from the above.
\begin{proposition}\label{prop:fibration}
The map $\widetilde q$ defined by the commutative diagram~\eqref{diagram}
is smooth and its
fibers are coassociative $4$-folds in $(\widetilde{M},\widetilde\varphi)$.
Every smooth fiber of $\widetilde q$ is diffeomorphic to $T^4$. The
singular fibers occur in $1$-parameter family, of codimension $2$
in~$\widetilde M$. The family of the smooth fibers of $q$ and the family of the singular fibers
of $q$ are each a maximal family of coassociative deformations.
\end{proposition}

{\small

\vspace{0.15cm}

\noindent{\sf M. Fern\'andez.} 
{Departamento de Matem\'aticas,
Facultad de Ciencia y Tecnolog\'{\i}a, 
Universidad del Pa\'{\i}s Vasco, Apartado 644,
48080 Bilbao, Spain. }\\
\hbox{\tt marisa.fernandez@ehu.es}

\vspace{0.15cm}

\noindent{\sf A. Fino.} {Dipartimento di Matematica \lq\lq Giuseppe Peano\rq\rq, 
Universit\`a di Torino, Via Carlo Alberto 10, 10123 Torino, Italy.}\\
\hbox{\tt annamaria.fino@unito.it}
 
 \vspace{0.15cm}
 
 \noindent{\sf A. Kovalev.} 
{DPMMS, University of Cambridge,\\
 Centre for \mbox{Mathematical} Sciences,\\
  \mbox{Wilberforce} Road, Cambridge CB3 0WB, UK }\\
\hbox{\tt A.G.Kovalev@dpmms.cam.ac.uk}

  \vspace{0.15cm}
  
\noindent{\sf V. Mu\~{n}oz.}
{Departamento de \'Algebra, Geometr\'{\i}a y Topolog\'{\i}a, Universidad de M\'alaga, 
Campus de Teatinos, s/n, 29071 M\'alaga, Spain.} \\
\hbox{\tt vicente.munoz@uma.es}

}

\end{document}